\documentclass[11pt]{article}

\usepackage{amsmath}
\usepackage{amssymb}
\usepackage{bbm}
\newtheorem{theorem}{Theorem}[section]
\newtheorem{proposition}[theorem]{Proposition}
\newtheorem{lemma}[theorem]{Lemma}
\newtheorem{corollary}[theorem]{Corollary}

\newtheorem{example}[theorem]{Example}

\newtheorem{remark}[theorem]{Remark}

\newenvironment{proof}{\medbreak\noindent{\it Proof.}\rm}{\hfill$\square$\rm}
 \newcommand{\beq}{\begin{equation}}
\newcommand{\eeq}{\end{equation}}

\numberwithin{equation}{section}

\newcommand{\R}{{ \mathbb R}}

\newcommand{\Rnp}{{\mathbb R}_+^n}
\newcommand{\Rnm}{{\mathbb R}_-^n}
\newcommand{\Q}{{ \mathbb Q}}

\newcommand{\C}{{\mathbb  C}}
\newcommand{\Cn}{{\mathbb  C}^n}

\newcommand{\D}{{\mathbb D}}

\newcommand{\A}{{\mathbb  A}}
\newcommand{\Bn}{{\mathbb B}^n}
\newcommand{\B}{{\mathbb B}}

\newcommand{\Bone}{\mathbbm{1}}

\newcommand{\cB}{{\mathcal B}}
\newcommand{\F}{{\mathcal F}}
\newcommand{\E}{{\mathcal E}}
\newcommand{\cG}{{\mathcal G}}
\newcommand{\K}{{\mathcal K}}

\newcommand{\N}{{\mathcal N}}
\newcommand{\I}{{\mathcal I}}
\newcommand{\cO}{{\mathcal O}}
\newcommand{\cP}{{\mathcal P}}
\newcommand{\cS}{{\mathcal S}}
\newcommand{\fb}{{\mathfrak b}}

\newcommand{\PSH}{{\operatorname{PSH}}}
\newcommand{\DSH}{{\operatorname{SH}}}
\newcommand{\SH}{{\operatorname{SH}}}

\newcommand{\codim}{{\operatorname{codim}}}

\newcommand{\Capa}{{\operatorname{Cap}\,}}
\newcommand{\Ree}{{\operatorname{Re}\,}}

\begin{document}

% APM setting
%\baselineskip=17pt

%\vskip1cm
\begin{center}
{\Large\bf Rooftop envelopes and residual plurisubharmonic functions}
\end{center}

\begin{center}
{\large\bf Alexander Rashkovskii}
\end{center}

\begin{abstract} Given a negative plurisubharmonic function $\phi$ in a bounded pseudoconvex domain of $\Cn$, we introduce and study its {\it residual function} $g_\phi$ determined by the asymptotic behavior of $\phi$ near its singularity points, both inside the domain and on its boundary. For certain choices of $\phi$, the function $g_\phi$ coincides with different versions of pluricomplex Green functions. The considerations are motivated by a problem on when two given plurisubharmonic functions can be connected by a plurisubharmonic geodesic.

\medskip\noindent
{\sl Mathematic Subject Classification}: 32U05, 32U15, 32U35, 32W20
\end{abstract}

\section{Introduction}
The standard pluricomplex function $G_a=G_{a,D}$ of a bounded pseudoconvex domain $D$ of $\Cn$ with pole at a point $a\in D$ is the upper envelope of all $u\in\PSH^-(D)$, that is, negative plurisubharmonic ({\it psh}) functions in $D$, such that $u(z)\le \log|z-a|+O(1)$ near $a$ \cite{Kl}. A generalized version due to Zahariuta  \cite{Za0} is obtained by replacing $\log|z-a|$ with a psh function $\phi$, locally bounded and maximal in a punctured neighborhood of $a$; one can also drop the maximality condition \cite{R06}, consider functions with finitely or denumerably many poles \cite{Za0}, \cite{Lel}, \cite{Ze}, \cite{R06}, and even with singularities along arbitrary analytic varieties \cite{LS}, \cite{RSig2}.
In all those constructions, the psh functions in question are subject to {\sl local} conditions near their singularity points. In this paper, we exploit a {\sl global} one.

Given $\phi\in\PSH^-(D)$, we introduce the function
\beq\label{GP0} g_{\phi}(z)=g_{\phi,D}(z)=\limsup_{x\to z}\sup\{v(x):\: v\in\PSH^-(D),\ v\le \phi + C_v\}.\eeq
This is a psh function determined essentially by the asymptotic behavior of $\phi$ near its unboundedness points. When $\phi=\log|z-a|$, $a\in D$, this gives us the pluricomplex Green function $G_a$. At the other extremity, if $D$ is the unit disk $\D\subset\C$ and $\phi$ is the negative Poisson kernel with pole at a boundary point,  then $g_\phi=\phi$. In the general case, the picture can be much more complicated. Since the singularities can lie both inside the domain and on its boundary, we call (\ref{GP0}) {\it the Green-Poisson residual function} of $\phi$ for the domain $D$.

Our considerations are motivated by a question of possibility of connecting pairs of psh  functions by psh geodesics, which reduces to the following
{\it Connectivity Problem: Given $u_0,u_1\in\PSH^-(D)\subset \Cn$, does there exist $u(z,\zeta)=u(z,|\zeta|)\in\PSH^-(D\times{\mathbb A})$, where ${\mathbb A}=\{\zeta\in\C: \, 1<|\zeta|<e\}$, such that $u(z,\zeta)\to u_j(z)$ as $|\zeta|\to e^j$, $j=0,1$? }

The origins lie in studying K\"{a}hler metrics on compact complex manifolds $(X,\omega)$. Such metrics are given as $e^{-u}$ for smooth quasi-psh functions $u$ on $X$ satisfying $\omega+dd^c u>0$, and geodesics in the space of the metrics correspond to the functions $u_t=u_{\log|\zeta|}$ on $X\times \mathbb A$, satisfying there the corresponding homogeneous complex Monge-Amp\`ere equation and having $u_0$ and $u_1$ as their boundary values. By arguments due to Berndtsson \cite{B15}, the Dirichlet problem has a unique solution; since {\sl a priori} it satisfies only $\omega+dd^cu_t\ge 0$ and need not be smooth, it is called a {\it weak geodesic}, and the corresponding metric might be singular. The arguments still work for arbitrary bounded quasi-psh data, however allowing singularities makes the problem highly non-trivial. It was shown in \cite{Da14a} that $u_0$ and $u_1$ can be connected if and only if $P_\omega[u_1](u_0)=u_0$ and $P_\omega[u_0](u_1)=u_1$ (and, in particular, this is the case when both the functions have full non-pluripolar Monge-Amp\`ere mass). Here, given a smooth closed $(1,1)$-form $\theta$, $P_\theta[u](v)$ is the {\it asymptotic rooftop of $u$ with respect to singularity $v$}, introduced (in the K\"{a}hler case $\theta=\omega$) in \cite{RWN}: denoting by $P_\theta(u,v)$ the {\it rooftop envelope} of $\theta$-psh functions $u$ and $v$ (i.e., the largest $\theta$-psh minorant of $\min\{u,v\}$),
$$P_\theta[v](u)={\sup_{C\in\R}}^* P_\theta(u,v+C),$$
where $\sup^*$ stands for the upper semicontinuous regularization of the envelope as in (\ref{GP0}); later on, we will use the similar denotation ${\lim}^*$.

This was continued in \cite{DNL18}--\cite{DNL19}, \cite{Mc}, \cite{McT}, \cite{RWN17}, and other recent papers. We would like especially to refer to \cite{DNL18a} where the following was proved: for $\theta$-psh functions $u$ and $v$ such that $\theta_v^n:=(\theta+dd^c v)^n$ has positive non-pluripolar mass and $u\le v+C$ (in other words, $u$ has a stronger singularity than $v$), the condition $P_\theta[u](v)=v$ is equivalent to the equality
$P_\theta[u](0)=P_\theta[v](0)$ (the {\it envelopes of the singularity types} of $u$ and $v$ coincide), as well as to the equality of their total non-pluripolar Monge-Amp\`ere masses: $\theta_u^n(X) =\theta_v^n(X)$. The proof rests heavily on a control over the total non-pluripolar Monge-Amp\`ere masses of $\theta$-psh functions, including their monotonicity.

In the local setting of psh functions on bounded domains  $D\subset\Cn$ (where the flat case $\theta=0$ is natural), the geodesics were considered in \cite{BB}, \cite{R16}, \cite{A}. In particular, Berndtsson's argument still works for bounded psh functions, and Darvas' construction from \cite{Da14a} was shown in \cite{R16} to work for functions $\phi$ from the Cegrell class $\F_1(D)$. We will recall definitions of this and other Cegrell's classes in Section~\ref{sec:Ceg}; here we just mention that any $\phi\in\F_1(D)$ is a negative psh function with well-defined Monge-Amp\`ere operator $(dd^c\phi)^n$, satisfying
\beq\label{F1} \int_D (1+|\phi|)(dd^c\phi)^n<\infty\eeq
(so the functions have finite both the total Monge-Amp\`ere mass and energy) and whose least maximal psh majorant in $D$ is identical zero.

 As we will see, there are two essential properties of functions $\phi\in\F_1(D)$ that make this happen: they do not have `bad' singularities on the boundary $\partial D$ of $D$, and
 the measures $(dd^c\phi)^n$  do not charge pluripolar sets. This turns out to be of crucial importance. It  was shown in \cite{R16} that if $u_0,u_1\in\PSH^-(D)$ have  strong singularities at isolated points (in the sense that their Monge-Amp\`ere measure charge the points) that are essentially different, then, for any corresponding subgeodesic, its limit values at the endpoints are strictly less than the corresponding data $u_0$ and $u_1$. For example, if $u_j=G_{a_j}$, then the largest subgeodesic is independent of $t$ and it equals the pluricomplex Green function $G_{\{a_0,a_1\}}$ with logarithmic poles at $a_0$ and $a_1$. Note that the classical pluricomplex Green function $G_a$ is, in the terminology of \cite{DNL18a}, the {\it envelope of singularity} $P_0[\phi](0)$ of $\phi(z)=\log|z-a|$, and the
 pluricomplex Green function $G_A$ for a finite set $A\subset D$ is exactly the envelope of the singularity of the function
$\phi(z)= \sum_{a\in A} \log|z-a|$.

While considering $\max\{\phi,\psi\}$ for psh functions $\phi$ and $\psi$ is standard in pluripotential theory, taking their rooftop envelope $P(\phi,\psi)=P_0(\phi,\psi)$, the largest psh minorant of $\min\{\phi,\psi\}$, is a perfectly natural operation that is, so far, insufficiently well studied.
For psh singularities, $\max\{\phi,\psi\}$ is a psh version of the greatest common divisor, and $P(\phi,\psi)$ corresponds to the least common multiple.
Together with addition, each of them generates an idempotent semiring of psh functions, which in the toric case is dual to a semiring of unbounded convex sets in $\Rnp$, see Example~\ref{indic}. Note also that the Green functions mentioned above are related by $G_{\{a_0,a_1\}}=P(G_{a_0},G_{a_1})$.

The asymptotic rooftop
$P[\phi](\psi)=P_0[\phi](\psi)$ of
$\psi$ with respect to singularity of $\phi$ is an extremal psh function for a Phragm\'{e}n-Lindel\"{o}f type problem in the spirit of \cite{MTV}:  the best upper bound on psh functions $v$ satisfying the pointwise inequality $v\le \psi$ and an asymptotic bound $v\le \phi+O(1)$.
In classical potential theory, similar asymptotic rooftops of subharmonic (or rather superharmonic) functions were considered by Parreau \cite{Pa} back in 1951 for a problem of approximation of unbounded positive harmonic functions by unbounded ones, and then in \cite{Y}, \cite{AL} and, for psh functions, in \cite{NW}; see the details in Remark~\ref{rem:parreau}.

By a minor adaption of Darvas' construction, one concludes that the geodesic connectivity of $u_0,u_1\in\PSH^-(D)$ can be checked by the conditions
\beq\label{Darv0} P[u_1](u_0)=u_0\quad  {\rm and}\quad P[u_0](u_1)=u_1,\eeq
see Theorem~\ref{prop_geod}. However, in contrast to the compact case, this is no longer equivalent to the equality of the total Monge-Amp\`ere masses of $u_0$ and $u_1$, either non-pluripolar or full (when the latter is well defined).
We would like then to compare (\ref{Darv0}) with the equality of the envelopes of singularity types, which in terms of the Green-Poisson functions (\ref{GP0}) reads as
\beq\label{eqs}g_{u_0}= g_{u_1};\eeq
its equivalence to (\ref{Darv0}) was conjectured by the author in a correspondence with G.~Hosono about an early version of \cite{H}. In that paper, it was actually shown that if $u_j$ are {\sl toric} psh functions equal to zero on $\partial D$, then relations (\ref{Darv0}) hold true if and only if all the directional Lelong numbers of $u_0$ are equal to those of $u_1$, which in this case, by \cite{R13b}, means exactly (\ref{eqs}). The proof in \cite{H} is based on a machinery of convex analysis and its applications to pluripotential theory indicated in \cite{Gu} and cannot be applied to the non-toric situation.

When $\phi\ge g_\phi+C$ (corresponding to the case of $\phi$ with {\it model singularity} \cite{DNL18a}), we have
\beq\label{main0} P[\phi](\psi) =P(\psi, g_\phi)\eeq
for any psh $\psi$, which gives us (\ref{Darv0}) for any model $u_0$ and $u_1$ satisfying (\ref{eqs}). It raises the question if (\ref{main0}) is true for non-model singularities as well.

With this perspective in mind, the main focus of the paper is on studying the Green-Poisson residual functions $g_\phi$ of negative psh functions $\phi$ on B-regular domains $D$ of $\Cn$. There are several challenges in such a setting distinguishing it from the one on compact manifolds. First, $\phi\in\PSH^-(D)$ can have its singularities on the boundary, and theory of boundary behavior of such psh functions is at the moment underdeveloped. Another issue is the lack of control over the total Monge-Amp\`ere mass of $\phi$ (even if $(dd^c\phi)^n$ is well defined) and of monotonicity for its non-pluripolar part, the central tools used in the compact case \cite{DNL18}--\cite{DNL19}.

It is easy to see that $g_\phi$ is maximal outside the unbounded locus $L(\phi)$ of $\phi$ and has zero boundary values outside the boundary unbounded locus $BL(\phi)$, however extending this to $D\setminus L(g_\phi)$ and $\partial D\setminus BL(g_\phi)$, respectively, is not that simple. By using a generalization of the classical comparison principle for bounded psh functions, we are able to do that in the case of $\phi$ with {\it small unbounded locus}, which in this setting means that both $L(\phi)$ and $BL(\phi)$ are pluripolar sets. Furthermore, we require the smallness also for proving the {\it idempotency} $g_{g_\phi}=g_\phi$ of the Green-Poisson function, a property of fundamental importance. Note that its compact counterpart, $P_\theta[P_\theta[u](0)](0)=P_\theta[u](0)$, is known to be true for any $\theta$-psh $u$ with positive non-pluripolar mass of $\theta_u^n$ \cite{DNL18a}.

We also compare the Green-Poisson function $g_\phi$ with the Green and Poisson functions $g_\phi^o$ and $g_\phi^b$, constructed from the asymptotic behavior of $\phi$ near $L(\phi)$ and $BL(\phi)$, respectively, and we relate the three residual functions to the pluricomplex Green functions with analytic singularities of positive dimension.

A bit more can be said in the case of $\phi$ from the Cegrell class $\E(D)$, which is the largest class of negative psh functions in the domain of definition of the Monge-Amp\`ere operator. We show that, in this case, $(dd^c g_\phi)^n$ is the residual part of $(dd^c\phi)^n$ at the set $\{ g_\phi=-\infty\}$. If, in addition, the least maximal psh majorant $\fb\phi$ of $\phi$ in $D$ is identical zero, then $g_\phi$ is idempotent and coincides with $g_\phi^o$, while $g_\phi^b=0$. When $\fb\phi\neq0$, it can be considered as a representation for the `boundary values' of $\phi$, like it is done by the least harmonic majorant of a subharmonic function in a domain of the complex plane. According to \cite{Ce08}, a function $\phi$ belongs to the class $\N(D,\fb \phi)$ if $\phi\ge \fb\phi+w$ for some $w\in\E(D)$ with $\fb w=0$; for example, this is so if $(dd^c\phi)^n(D)<\infty$.  For such $\phi$, we show that $\fb g_\phi=g_{\fb\phi}= g_\phi^b$ and $g_\phi\in\N(D,\fb g_\phi)$. Furthermore, if $\fb\phi$ has small unbounded locus, then $(dd^c g_\phi^o)^n= (dd^c g_\phi)^n$ and if, in addition, $\phi$ has finite total residual Monge-Amp\`ere mass, then  $g_\phi=P(g_\phi^o, g_\phi^b)$,  which is a nonlinear analog of the Poisson-Jensen formula for the residual functions.

Concerning the asymptotic rooftop envelopes $P[\phi](\psi)$, we have evidently $$P(\phi,\psi)\le P[\phi](\psi) \le P(\psi, g_\phi),$$ which gives us the corresponding relations for their residual functions. Moreover, when $\phi$ and $\psi$ have small unbounded loci, it is easy to see that, actually,
the residual functions of the first two items here coincide, however we do not know if the equality extends to that for the last one, which would be a necessary condition for (\ref{main0}).

We are interested if (\ref{main0})
holds true for all $\phi,\psi\in\PSH^-(D)$ because it would guarantee the implication  (\ref{eqs}) $\Rightarrow$ (\ref{Darv0}). By the reasons explained above, the question turns out to be tricky already for $\phi,\psi\in\E(D)$ with small unbounded loci.
With no counterexamples in hand, we were able to establish it only in a few cases, apart from the obvious one of $\phi$ with model singularity. For example, (\ref{main0}) is true, provided $\phi\ge g_\phi+w$ such that $g_w=0$; this is so, in particular, if $g_\phi=0$, which already handles the case $\phi\in\F_1(D)$ without referring to the Monge-Amp\`ere technique used for proving this in \cite{R16}. Another slight relaxation of the model condition is replacing it with $ g_\phi(z)/\phi(z)\to 1$ as $\phi(z)\to -\infty$, in which case we say that the singularity is {\it approximately model}. In particular, any asymptotically analytic singularity in the sense of \cite{R13a} is approximately model in every $D'\Subset D$ (and in $D$, provided $BL(\phi)=\emptyset$).

We apply this to the connectivity problem for $u_0,u_1\in\PSH^-(D)$, described in the beginning, which, as already said, is equivalent to conditions (\ref{Darv0}). In particular, it is shown that no pair of functions with small unbounded loci and different Green-Poisson functions can be geodesically connected. On the other hand, assuming both $u_0$ and $u_1$ satisfy one of the conditions from the previous paragraph, the equality $g_{u_0}= g_{u_1}$ implies the connectivity.

\medskip
Here is a short summary of the main results of the paper.

(i) The residual Green-Poisson function $g_\phi$ is introduced for arbitrary negative psh functions $\phi$ on bounded B-regular domains of $\Cn$. For $\phi$ with small unbounded loci $L(\phi)\subset D$ and $BL(\phi)\subset\partial D$, it is shown in Theorem~\ref{gphi2} to be maximal outside $L(g_\phi)$ and to have zero boundary values outside $BL(g_\phi)$, as well as to be idempotent: $g_{g_\phi}=g_\phi$. To prove this, we establish in Lemma~\ref{CP} a stronger version of the classical domination principle for bounded psh functions by relaxing the boundary conditions.

(ii) In the case of $\phi$ in the domain $\E$ of definition of the Monge-Amp\`ere operator, we show that $(dd^cg_\phi)^n$ is the residual part of $(dd^c\phi)^n$ on $\{\phi=-\infty\}$ (Proposition~\ref{gphiphi}), while the boundary value function $\fb g_\phi$ of $g_\phi$, in the sense of Cegrell, equals $g_{\fb\phi}$ (Theorem~\ref{bd}).

(iii) When $\phi\in\E$ has finite total Monge-Amp\`ere mass and $\fb\phi$ has small unbounded locus, $g_\phi$ is shown to be the rooftop envelope of the residual functions constructed separately by the singularities of $\phi$ inside $D$ and on $\partial D$ (Corollary~\ref{genv}). In addition, if a sequence $\phi_j$ of such functions with uniformly bounded Monge-Amp\`ere masses increases q.e. to $\phi$, then $g_{\phi_j}$ increase to $g_\phi$ (Theorem~\ref{gincr}).
Proofs of the results in (ii) and (iii) are based on the technique of Cegrell classes.

(iv) We relate the asymptotic rooftop envelopes $P[\phi](\psi)$ to the envelopes $P(\psi,g_\phi)$. While the left hand side does not exceed the right hand side for any $\phi,\psi\in\PSH^-(D)$, the reverse inequality we are able to establish only in few cases: when $\phi\ge g_\phi + w$ with psh $w$ satisfying $g_w=0$, or when $\phi$ has approximately model singularity (Proposition~\ref{propyes}), which already implies a new result for functions from the Cegrell class $\F^a$ (Remark~\ref{remyes}). The principal  challenge in treating the general situation is the absence of comparison of non-pluripolar mass in the local setting.

(v) Finally, we show that the equality $g_\phi=g_\psi$ is a necessary condition for existence of a psh (sub)geodesic connecting $\phi$ and $\psi$ when they have small unbounded loci (Corollary~\ref{corind1}), while it is also a sufficient condition in the cases treated by Proposition~\ref{propyes} (Theorem~\ref{last}). In addition, any $\phi$ can be geodesically connected with its residual function $g_\phi$. This part uses the previous results and an adaptation of technique from  \cite{Da14a}.

\medskip
The presentation is organized as follows. In Section 2, we recall basic properties of the rooftop envelopes $P(u,v)$. In Section~3, we introduce the residual Green-Poisson function $g_\phi$ and study its general properties, including the boundary behavior, idempotency, and interaction with psh structural operations. In Section~4, we compare it with the Green and Poisson functions $g_\phi^o$ and $g_\phi^b$, and relate it to the pluricomplex Green functions with analytic singularities of positive dimension. In Section~5, we consider classes of approximately model and asymptotically analytic singularities. A more detailed information on the residual functions is presented in Section~6 for the functions from Cegrell's classes. In Section~7, we consider the asymptotic rooftop envelopes $P[\phi](\psi)$ and find sufficient conditions for (\ref{main0}). We apply this to the geodesic connectivity problem in Section~8. The final Section~9 presents some open questions on the subjects of the paper. Here we mention just one of them: if $\phi_j\nearrow \phi$, is it true that $P(\phi_j-\phi)\nearrow 0$? This seems to be unknown even for $n=1$, which shows how little we know about such envelopes at all.

\section{Rooftop envelopes}

Rooftop envelopes were explicitly introduced in \cite{RWN} for quasi-psh functions on compact K\"{a}hler manifolds, and
in the local context they were considered in \cite{R16} for functions in the Cegrell class $\F_1$. Here we will be interested in a more general settings of bounded from above psh functions.

In this section, we fix a bounded pseudoconvex domain $D\subset\Cn$. Given a function $h$ on $D$, let $\cB(h,D)= \{v\in\PSH(D):\: v\le h\}$ and
$$P(h)=P_D(h)={\sup}^*\{v\in\cB(h,D)\},$$
that is, the u.s.c. regularization of the function $\hat h={\sup}\,\{v\in \cB(h,D)\}$.
When $h$ is locally bounded from above, the function $P(h)$ is psh and called the {\it psh envelope} of $h$  in $D$. Here are some its elementary properties; for bounded $h$, they were proved in \cite{GLZ}, however the proofs work as well for all $h$ bounded from above, and we present them here for completeness.

\begin{proposition}\label{Ph} Let $h, h_1, h_2, \ldots$ be measurable, locally bounded from above functions on $D$, such that $P(h), P(h_j)\not\equiv -\infty$. Then
\begin{enumerate}
\item[(i)] $P(h)\in\PSH(D)$;
\item[(ii)]  $P(h)\le h$ q.e. (quasi everywhere, i.e., outside a pluripolar subset) in $D$;
\item[(iii)] $P(h)$ equals the upper envelope of the class $\cB^*(h,D)$ of all functions $v\in\PSH(D)$ such that $v\le h$ q.e. in $D$;
\item[(iv)] if $h_j$ decrease to $h$, then $P(h_j)$ decrease to $P(h)$.
\end{enumerate}
\end{proposition}

\begin{proof}
Assertions (i) and (ii) are standard facts of pluripotential theory, see \cite{BT82}.

Using Choquet's lemma and pluripolarity of unions of countably many pluripolar sets, the function $\phi=\sup\,\{v\in\cB^*(h,D)\}$ belongs to the class $\cB^*(h,D)$; in particular, $\phi\in\PSH(D)$ and $\phi\ge P(h)$. Since the set
$$ E=\{z:\: \phi(z)>P(h)(z)\}\cup \{z:\: P(h)(z)>h(z)\} $$
is pluripolar, there exists $v\in\PSH(D)$, not identically $-\infty$ and such that $v=-\infty$ on $E$. By \cite[Thm. 5.8]{Ce04}, one can assume $v<0$. Then, for any $\epsilon>0$, we have
$\phi+\epsilon\,v\le h$ everywhere in $D$, so $\phi+\epsilon\,v\le P(h)$. By taking $\epsilon\to 0$, we get $\phi\le P(h)$ quasi everywhere and thus everywhere in $D$, which proves (iii).

Finally, if $h_j$ decrease to $h$, then $P(h_j)$ decrease to a psh function $u\ge P(h)$. By (ii), $u\le h_j$ quasi everywhere. Since the union of countably many pluripolar sets is pluripolar, we get $u\le h$ outside a pluripolar set $E$ and thus, by (iii), $u\le P(h)$, which gives us (iv).
\end{proof}

\begin{remark}{\rm It was shown in \cite{GLZ} that, if {\sl bounded } $h_j$ increase to $h$, then $P(h_j)$ increase to $P(h)$ quasi everywhere. It is easy to see that this is no longer true if $h\le 0 $ is {\sl unbounded}, even in simple rooftop situations, see Example~\ref{exincr}.}
\end{remark}

Given $u,v\in \PSH(D)$, denote
$$ P(u,v)=P(\min\{u,v\}),$$
{\it the rooftop envelope} of $u$ and $v$ (we will use this notation also in the case of arbitrary bounded above functions $u$ and $v$).
For such a case, Proposition~\ref{Ph} adjusts as follows.

\begin{proposition}\label{Puv} Let $u,v, v_1, v_2, \ldots\in\PSH(D)$. Then
\begin{enumerate}
\item[(i)] $P(u,v)\in\PSH(D)$; if  $u,v\in\PSH^-(D)$, then $P(u,v)\ge u+v$;
\item[(ii)]  $P(u,v)\le \min\{u,v\}$ everywhere on $D$ (in other words, it is the largest psh minorant of $\min\{u,v\}$);
\item[(iii)] if $v_j$ decrease to $u$, then $P(u,v_j)$ decrease to $P(u,v)$.
\end{enumerate}
\end{proposition}

\begin{remark}\label{rem:max} {\rm 1. Note that the inequality $P(u,v)\ge u+v$ makes always sense for psh functions, contrary to the case of quasi-psh functions where the sum need not be quasi-psh.

2. By (ii), maximality of $u$ and $v$ on $D'\subset D$ implies maximality of $P(u,v)$.

3. Moreover, as follows from \cite[Prop. 3.3]{Da14a} (see also \cite[Lemma 3.7]{DNL18a}),
\beq\label{MAP} {\rm NP}(dd^c[P(u,v)])^n\le \Bone_{\{P(u,v)=u\}} {\rm NP}(dd^c u)^n + \Bone_{\{P(u,v)=v\}} {\rm NP}(dd^c v)^n, \eeq
where ${\rm NP}(dd^c w)^n$ is the {\it non-pluripolar Monge-Amp\`ere operator} in the sense of \cite{BT87}: for Borel sets $E$,
$${\rm NP}(dd^c w)^n=\lim_{j\to\infty}\Bone_{E\cap\{w>-j\}} (dd^c\max\{w,-j\})^n.$$
}
\end{remark}

The function $P(u,v)$ is a psh version of the notion of least common multiple. The extreme cases are $P(u,v)=\min\{u,v\}$ (and then either $v\le u$ or $u\le v$) and $P(u,v)= u+v$; in the latter situation, we will say that $u$ and $v$ are {\it relatively prime} in $D$.

\begin{example} {\rm The functions $u=\log|z|$ and $v=-1$ in the unit ball $\Bn$ are relatively prime. Indeed, let $w\le P(u,v)$, then $w_1:=w+1\in\PSH^-(\Bn)$ satisfies $w_1\le u+ 1$. Therefore, it is dominated by the pluricomplex Green function for $\Bn$ with pole at $0$, that is, by $u$. This gives us $w\le u+v$.
}
\end{example}

Another example of relatively prime functions are $\log|z_1|$ and $\log|z_2|$ in the bidisk. More generally, in the analytic case, we have the following

\begin{proposition}\label{analenv} If $f_j=fh_j\in\cO(D)$, $j=1, 2$, and $\codim\{z:\: h_1(z)=h_2(z)=0\}>1$, then
$P(\log|f_1|,\log|f_2|)= \log|fh_1h_2|+ v$, where $v\ge0$ is a maximal psh function in $D$.
\end{proposition}

\begin{proof}  The function $v:= P(\log|h_1|,\log|h_2|)- \log|h_1h_2|$ is non-negative and psh on $D\setminus Z$. Since $\codim\, Z>1$, it extends to a non-negative psh function on $D$. By Remark~\ref{rem:max}.2, it is maximal on $D\setminus Z$ and, therefore, on $D$. Finally,
$$P(\log|f_1|,\log|f_2|)= \log|f|+P(\log|h_1|,\log|h_2|)= \log|f| + \log|h_1h_2|+ v,$$
which proves the claim.
\end{proof}

\medskip
Unlike the continuity under decreasing limit transitions given by Proposition~\ref{Puv}(iii), the behavior of $P(u,v_j)$ with increasing $v_j$ can be more complicated, provided $v_j$ are unbounded from below.

\begin{example}\label{exincr} {\rm Let $D=\Bn$, $u=0$, $v_j=\max_k\log|z_k|+j$. Then $\min\{u,v_j\}$ increase, as $j\to\infty$, to the function $\hat h$ equal to $0$ outside the origin and $\hat h(0)=-\infty$, while $P(u,v_j)$ increase to $\log|z|$ which has the same singularity as $\max_k\log|z_k|$.}
\end{example}

The following two simple technical observations are sometimes useful.

\begin{proposition}\label{Phu} If $h$ is a measurable, bounded from above function on $D$, then $P(P(h),u)=P(h,u)$ for any $u\in\PSH(D)$.
\end{proposition}

\begin{proof} The inequality $P(h)\ge P(h,u)$ gives us $P(P(h),u)\ge P(h,u)$. To prove the reverse, we get, as in the proof of Proposition~\ref{Ph}, $v\in\PSH^-(D)$, $v\not\equiv -\infty$, equal to $-\infty$ on the set where $P(h)>h$. Then, for any $\epsilon>0$, we have
$$ P(P(h),u)+\epsilon v \le P(P(h)+\epsilon v,u)\le P(h+\epsilon v,u)\le P(h,u).$$
Letting $\epsilon\to 0$ we get $P(P(h),u)\le P(h,u)$ quasi everywhere and thus everywhere on $D$.
\end{proof}

\begin{proposition}\label{PuP} If $u\in\PSH^-(D)$ and $v\in\PSH(D)$, then
$P(u,v+\alpha)=P(u,P(0,v+\alpha))$
for any measurable function $\alpha$ on $D$.
\end{proposition}

\begin{proof}
Indeed, $P(u,v+\alpha)\ge P(u,P(0,v+\alpha))$ because $v+\alpha\ge P(0,v+\alpha)$, and $P(u,v+\alpha)\le P(u,P(0,v+\alpha))$ because
$P(u,v+\alpha)=P(u,P(v+\alpha))\le P(u,P(0,v+\alpha))$.
\end{proof}

\medskip
We illustrate the notion of rooftop envelopes by considering a specific class of functions.

\begin{example}\label{indic} {Rooftops of indicators.} {\rm Let  $\cG$ be the collection of convex subsets $\Gamma$ of the positive orthant $\Rnp$, satisfying $\Gamma+\Rnp\subset\Gamma$. The support function
$$\psi_\Gamma(t)=\sup\{\langle a,t\rangle:\: a\in \Gamma\}$$
of $\Gamma\in\cG$
is a negative convex function on  $\Rnm=-\Rnp$, increasing in each component $t_j$, and positively homogeneous: $\psi_\Gamma(ct)=c\,\psi_\Gamma(t)$ for any $c>0$. Then its psh image $$\Psi_\Gamma(z):=\psi_\Gamma(\log|z_1|,\ldots,\log|z_n|)$$
extends to a negative psh function in the unit polydisk $\D^n$, an {\it indicator}. The least indicator dominating a function $u\in\PSH^-(\D^n)$ is $$\Psi_u(z)={\lim_{m\to\infty}}^*\, \frac1m \,u(z_1^m,\ldots,z_n^m),$$ the {\it indicator of} $u$, used in Kushnirenko-Bernshtein type bounds for the residual Monge-Amp\`ere mass at $0$ \cite{LeR}, \cite{R00}, \cite{R13b}. Note that its value at $z$ with $z_k=e^{-a_k}$, $a_k>0$, is the negative directional Lelong number of $u$ in the direction $(a_1,\ldots,a_n)$.

It is easy to see that
$ P(\Psi_{\Gamma_1}, \Psi_{\Gamma_2})=\Psi_{\Gamma_1\cap\Gamma_2}$.
This gives us a $(\min,+)$-tropical semiring of the indicators with operations of rooftop envelopes and addition, isomorphic to that of the sets in $\cG$ with operations of intersection and Minkowski's addition.
Note that one gets a $(\max,+)$-tropical semiring on $\cG$ considered in \cite{R09} by replacing the intersection with taking convex hull of the union.}
\end{example}

\section{Green-Poisson residual functions}\label{sec:GP}
 Given $\phi\in\PSH^-(D)$, let $L(\phi)$ denote its {\it unbounded locus}, i.e., the set of points $a\in D$ such that $u\not\in L_{loc}^\infty(a)$, and $BL(\phi)$ be its {\it unbounded boundary locus}, the set of points $b\in\partial D$ such that $u\not\in L^\infty(\omega\cap D)$ for any neighbourhood $\omega$ of $b$. Note that $L(\phi)\subset D$ is relatively closed and $BL(\phi)\subset \partial D$ is closed. We will often work with the functions that have {\it small unbounded locus} in the sense that both $L(\phi)$ and $BL(\phi)$ are pluripolar sets; the class of such functions will be denoted by $\PSH_s^-(D)$.

 From now on, we assume the domain $D$ to be B-regular, meaning that for any continuous function $\eta$ on $\partial D$ there exists a psh function in $D$, continuous on $\overline D$ and equal to $\eta$ on $\partial D$.
 By \cite{Si}, B-regularity is equivalent to saying that $D$ has a strong psh barrier at any its boundary point $p$ (i.e., there exists a function $\rho_p\in\PSH(D)$ such that $\rho_p(x)\to 0$ as $x\to p$ and $\sup_{D\setminus U}\rho_p<0$ for any neighbourhood $U$ of $p$). We need this property to guarantee that the functions $g_\phi$, constructed below, have zero boundary values, possibly apart from $BL(\phi)$. Actually, if $L(\phi)\Subset D$, B-regularity can be replaced by a weaker condition of {\it hyperconvexity}, i.e., that there exists a negative psh exhaustion function on $D$.

Let $\phi,\psi\in\PSH^-(D)$. We say that $\phi$ has {\it  stronger singularity} than $\psi$ in $D$ and denote $\phi \preceq_D \psi$ if $\phi(z)\le \psi(z) + C$ for some $C\in\R$ and all $z\in D$. We also say that the functions have equivalent singularities, $\phi \sim_D \psi$, if $\phi \preceq_D \psi$ and $\psi \preceq_D \phi$. When the domain $D$ is fixed, we will use just the symbols $\preceq$ and $\sim$.

 Given $\phi\in\PSH^-(D)$, let $\cS_\phi=\cS_{\phi,D}$ denote the class of functions with singularities at least as strong as that of $\phi$:
$$\cS_\phi=\{w\in\PSH^-(D): w\preceq \phi \}.$$
The function $$g_\phi(z)=g_{\phi,D}(z)={\sup}^*\{w(z):\: w\in\cS_{\phi,D}\}$$
will be called the {\it Green-Poisson residual function} for the singularity $\phi$. The term  reflects the fact that such a function is determined by the singularities of $\phi$ both inside the domain and near its boundary, see a discussion below and especially Example~\ref{Poisson}.
Evidently, $g_\phi$ equals the regularized limit of the rooftop envelopes $P(\phi+C,0)$ as $C\to \infty$:
\beq\label{rtlim} g_\phi(z)={\sup_{C}}^*P(\phi+C,0)= {\lim}^*_{C\to\infty}\,P(\phi+C,0).
\eeq

\begin{remark}\label{rem:parreau} {\rm A very close notion was introduced and studied in classical potential theory, starting with \cite{Pa} where the bounded approximations of positive harmonic functions were considered. To stick with our objects, let us assume $D$ to be a domain in $\C$. If $u$ is the Poisson integral $P[\nu]$ of a negative measure $\nu$ on $\partial D$ (or, more generally, the Martin integral of a negative measure on the Martin boundary of $D$), then $g_u=P[\nu_s]$ for the singular part of $\nu$ with respect to the harmonic measure. In \cite{AL}, it was shown that if $u$ is the classical subharmonic Green potential of a positive measure $\mu$ in a domain of $\C$, then $g_u$ is the potential $G_{\mu_s}$ of the restriction $\mu_s$ of $\mu$ to $\{u=-\infty\}$. Furthermore, any negative subharmonic function $u$ in $D$ represents as $u=u_s+u_t$ with $g_{u_s}=u_s$ (singular part) and $g_{u_t}=0$ (quasi-bounded, or tame, part).

For psh functions, the picture is more complicated. The corresponding notions of singular and tame, in the above sense, psh functions were considered in \cite{NW} and related to the problem of approximation of unbounded (from above) psh functions by the bounded ones.}
\end{remark}

The condition of uniform commensurability $u\sim v$ means that, in general, both the singularities inside the domain and at its boundary are taken into account.
In pluripotential theory, more standard and much better studied is considering extremal psh functions determined by singularities inside the domain. We start with the easiest case of functions with isolated singularities.

\begin{example}\label{pcgf} {\rm

1. When $\phi(z)\sim\log|z-a|$ with $a\in D$, the function $g_\phi$ is the classical pluricomplex Green function $G_{a}$ of $D$ with pole at $a$.

2. Similarly,  $\phi(z)\sim\sum_1^k m_j\log|z-a_j|$ generates a weighted multipole pluricomplex Green function.

3. More generally, replacing each $m_j\log|z-a_j|$ with a function $\phi_j\in\PSH^-(D)$ which has isolated singularity at $a_j$, is maximal on a punctured neighborhood of $a_j$ and bounded near $\partial D$, we get Zahariuta's Green function $G_{(\phi_j)}$ for the {\it maximal singularities} $\phi_1,\ldots,\phi_n$ \cite{Za0}, \cite{Za}.
This was extended to arbitrary (non-maximal) isolated singularities in \cite{R06} as the {\it greenification} of $(\phi_j)$.
}
\end{example}

Pluricomplex Green functions with purely boundary singularities will be considered later in Example~\ref{Poisson}, and with `fat' (non-discrete) singularities -- in Example~\ref{Gball}.

\medskip

Here are some elementary properties of the Green-Poisson residual functions.

\begin{proposition}\label{gphi} Let $\phi,\psi\in \PSH^-(D)$. Then
\begin{enumerate}
\item[(i)] $g_{c\,\phi}=c\,g_\phi$ for any $c>0$;
\item[(ii)] if $\phi\preccurlyeq\psi$, then $g_\phi\le g_\psi$;
\item[(iii)] $g_{\phi+\psi}\ge g_\phi+g_\psi$;
\item[(iv)] $g_{\max\{\phi,\psi\}}\ge \max\{g_\phi,g_\psi\}$;
\item[(v)] $g_{P(\phi,\psi)}\le P(g_\phi,g_\psi)$.
\end{enumerate}
\end{proposition}

By (\ref{rtlim}) and (\ref{MAP}), $(dd^c P(\phi+C,0))^n=0$ on $\{\phi> -C\}$, which implies

\begin{proposition}\label{NPMAg}
If $\phi\in\PSH^-(D)$, then ${\rm NP}(dd^c g_\phi)^n=0.$ In particular, $g_\phi$ is maximal on $D\setminus L(\phi)$.
\end{proposition}

Evidently, $\phi\le g_\phi$ for any $\phi\in \PSH^-(D)$, however the singularities of the two functions can be different. In particular, while $L(g_\phi)\subset L(\phi)$ and $BL(g_\phi)\subset BL(\phi)$, one can have
$L(g_\phi)\neq L(\phi)$ and $BL(g_\phi)\neq BL(\phi)$, which results in certain difficulties in handling these functions. This is one of the reasons of restricting here to the class $\PSH_s$ of functions with small unbounded locus. Even a more challenging issue is the important {\it idempotency property}
$$g_{g_\phi}=g_\phi$$ which at the moment we can prove only for $\phi\in \PSH_s^-(D)$, see Theorem~\ref{gphi2} below, whose proof rests on the following (probably, known) version of the classical domination principle.

\begin{lemma}\label{CP} If $u,v\in\PSH(\omega)\cap L^\infty(\omega)$ on $\omega\Subset\Cn$ satisfy
$(dd^c v)^n \le (dd^c u)^n$ and
\beq\label{genBV} \limsup_{z\to\zeta}(u(z)-v(z))\le 0\quad \forall \zeta\in\partial \omega\setminus F
\eeq
for a pluripolar set $F\subset\partial \omega$, then $u\le v$ in $\omega$.
\end{lemma}

\begin{proof} As in the classical case $F=\emptyset$ \cite{BT82}, the claim will follow from the corresponding comparison theorem: if (\ref{genBV}) is fulfilled, then
\beq\label{genDP}
\int_{\{v<u\}}(dd^cu)^n \le \int_{\{v<u\}}(dd^cv)^n.\eeq
To prove (\ref{genDP}), we note that the condition (\ref{genBV}) was used in the proof of the classical comparison theorem \cite[Thm.~4.1]{BT82} only in a reduction to the case $\{v<u\}\Subset \omega$ by replacing $u$ with $u-\delta$, $\delta\searrow 0$. In our situation, $u$ is to be replaced with the functions $u_\delta= u+\delta(\Psi-1)$, where $\Psi\in\PSH^-(\Omega)$ for a neighbourhood $\Omega$ of $\overline \omega$,  $\Psi\not\equiv -\infty$ and $\Psi=-\infty$ on  $F$. By \cite{Ce04}, one can assume $\Psi$ to belong to the Cegrell class $\F_1(\Omega)$, which implies that the $(dd^c\Psi)^n$ is well defined and does not charge pluripolar sets.

Then $(dd^c v)^n \le (dd^c u_\delta)^n$ on $\omega$ and $u_\delta$ is bounded on the set $\{v<u_\delta\}\Subset \omega$ converging to $\{v<u\}\setminus \Psi^{-1}(-\infty)$ as $\delta\to 0$. By \cite[Lem.~4.4]{Ce98},
$$  \int_{\{v<u_\delta\}}(dd^cu_\delta)^n \le \int_{\{v<u_\delta\}}(dd^cv)^n,$$
and (\ref{genDP}) follows because
$$ \int_{\{v<u_\delta\}}(dd^cu)^n \le \int_{\{v<u_\delta\}}(dd^cu_\delta)^n$$
and $(dd^c u)^n(\Psi^{-1}(-\infty))=0$.

Now, take $\psi(z)=|z|^2-C<0$ on $\omega$. If $\{v<u\}\neq\emptyset$, then $S=\{v<u+\epsilon\psi\}\neq\emptyset$ for some $\epsilon>0$ as well and has positive Lebesgue measure, $M$. Since the functions $\tilde u=u+\epsilon\psi$ and $v$ still satisfy the conditions of the lemma, (\ref{genDP}) gives us
$$ \epsilon^n\,M + (dd^cu)^n(S)\le (dd^c\tilde u)^n(S)\le (dd^cv)^n(S),$$
which contradicts $(dd^c v)^n \le (dd^c u)^n$.
\end{proof}

\begin{theorem}\label{gphi2} Let $\phi\in \PSH_s^-(D)$. Then
\begin{enumerate}
\item[(i)]
$g_\phi$ is a maximal psh function outside its unbounded locus $L(g_\phi)$;
\item[(ii)] $g_\phi=0$ on $\partial D\setminus BL(g_\phi)$;
\item[(iii)] $g_{g_\phi}=g_\phi$;
\item[(iv)] if $g_{\phi}'$ denotes the Green-Poisson function $g_{\phi,D'}$ of $\phi$ with respect to a domain $D'\subset D$, then $g_{g_\phi}'=g_\phi'$. As a consequence, if $g_\phi=g_\psi$ for a function $\psi\in \PSH_s^-(D)$, then $g_\phi'=g_\psi'$
 \end{enumerate}
\end{theorem}

\begin{proof} Maximality of $g_\phi$ outside $L(\phi)$ is established in Proposition~\ref{NPMAg}. Alternatively, we can use a more elementary, standard approach by using Perron-type arguments. By the Choquet lemma, there exists a sequence $u_j\in \cS_{\phi}$ increasing quasi everywhere to $g_\phi$. Take any open set $D'\Subset D\setminus L(\phi)$ and $\tilde u_j\in \PSH^-(D)$ equal to $u_j$ on $D\setminus D'$ and satisfying $(dd^c\tilde u_j)^n=0$ in $D'$. Then $u_j\in \cS_\phi$ increase quasi everywhere to $g_\phi$ as well and so, the latter satisfies  $(dd^c g_\phi)^n=0$ in $D'$ and, therefore, in $ D\setminus L(\phi)$.
To extend the maximality to $ D\setminus L(g_\phi)$, we use the condition that $\phi$ (and thus $g_\phi$) has small unbounded locus.  Since $g_\phi$ is locally bounded on  $ D\setminus L(g_\phi)$, the Monge-Amp\`ere measure $(dd^c g_\phi)^n$ cannot charge the pluripolar set $L(\phi)\setminus L(g_\phi)$, which gives us the maximality of $g_\phi$ in $ D\setminus L(g_\phi)$ and proves (i).

Assertion (ii) will be also proved in two steps. First, the relation $g_\phi=0$ on $\partial D\setminus BL(\phi)$ follows by standard arguments using the maximality of $g_\phi$ (see, for example, \cite[Prop.2.4]{LS}); in this part, no condition of small unbounded locus is needed. Namely, let $p\in\partial D\setminus BL(\phi)$, then $\phi> -K$ near $D\cap {\overline U}$ for some $K>0$ and a neighbourhood $U$ of $p$. Let $\rho_p$ be a strong psh barrier for $D$ at $p$, then $\sup_{D\setminus U}\rho_p<-K/c$ for some $c>0$ and so,
$\phi>c\,\rho_p$ near $\partial U\cap D$. Then the function $u$ equal to $\max\{\phi, c\rho_p\}$ on $D\cap U$ and to $\phi$ on $D\setminus U$ is psh in $D$ and belongs to $\cS_{\phi}$, while $u(x)\to 0$ as $x\to p$.

To extend this to $\partial D\setminus BL(g_\phi)$,  we choose a neighbourhood $U$ of a point $\zeta\in BL(\phi)\setminus BL(g_\phi)$ such that $U\cap\overline{L(g_\phi)}=\emptyset$ (which is possible because $g_\phi$ is bounded near $\zeta$), and apply Lemma~\ref{CP} in $\omega=U\cap D$ to $v=g_\phi$ and $u$ the solution to the Dirichlet problem for the homogeneous Monge-Amp\`ere equation with boundary value $g_\phi$ on $\partial \omega\cap D$ and $0$ on $\partial \omega\cap \partial D$.

To prove (iii), take any $u\in\cS_{g_\phi}$, then $u\le g_\phi +C$ for some $C>0$; clearly, we can assume $u\ge g_\phi$. For any $\epsilon>0$, let $N>C/\epsilon$, then $u\le (1-\epsilon)g_\phi$  in  $D_N=\{z\in D:\: g_\phi(z)<-N\}$. Take a function $\Psi\in\PSH^-(\Omega)$, $D\Subset\Omega$, $\Psi\not\equiv -\infty$, equal to $-\infty$ on $L(g_\phi)\cup BL(g_\phi)$. We the have $$\{z\in D:\: u(z)+\epsilon \Psi(z)>(1-\epsilon)g_\phi(z)\}\Subset D\setminus L(g_\phi)$$
(we add $\epsilon\Psi$ to take care of the approaching $L(g_\phi)\cup BL(g_\phi)$ from the outside of $D_N$).
Since the function $(1-\epsilon)g_\phi$ is maximal outside $L(g_\phi)$, this implies $u+\epsilon \Psi\le (1-\epsilon)g_\phi+\epsilon$ in $D\setminus L(g_\phi)$ and so,
$u= g_\phi$ in $D$.

 Finally, since $\phi\le g_\phi\le g_\phi'$ in $D'$, we have, by (iii),
 $$ g_\phi'\le g_{g_\phi}'\le g_{g_\phi'}'= g_\phi',$$
 which proves (iv).
\end{proof}

\medskip
In the next three propositions, we indicate interactions of the Green-Poisson residual functions with sums, maxima and rooftop envelopes of psh functions.

\begin{proposition}\label{gphisum} Let $\phi,\psi\in \PSH_s^-(D)$. Then
\begin{enumerate}
\item[(i)] $g_{\phi+\psi}=g_{g_\phi+g_\psi}$;
\item[(ii)]
$g_{\phi+\psi}=g_\phi$ iff $g_\psi=0$.
 \end{enumerate}
\end{proposition}

\begin{proof} Relation (i) follows, by Theorem~\ref{gphi2}(iii), from
$ g_{\phi+\psi}\ge g_\phi+g_\psi\ge \phi+\psi$.
Next, let $g_\psi=0$, then, by (i) and Proposition~\ref{gphi}(ii),
$$ g_\phi\ge g_{\phi+\psi}\ge g_\phi+g_\psi = g_\phi,$$
which gives $g_{\phi+\psi}=g_\phi$.
Conversely, if $g_{\phi+\psi}=g_\phi$, then
$ g_\phi\le g_\psi$. Furthermore, since $\phi\le g_\phi$ and $\psi\le g_\psi$, we have
$$ g_\phi=g_{\phi+\psi}= g_{g_\phi+g_\psi}\le 2g_\psi.$$
By repeating this, we get $g_\phi\le Ng_\psi$ for each positive $N$ and thus, $g_\psi=0$.
\end{proof}

\begin{proposition}\label{gmax}
If $\phi, \psi\in\PSH_s^-(D)$, then
\begin{enumerate}
\item[(i)] $g_{\max\{\phi,\psi\}}=g_{\max\{g_\phi,g_\psi\}}$;
\item[(ii )]  $g_{\max\{\phi,\psi\}}=g_\phi$ iff $g_\psi\le g_\phi$.
 \end{enumerate}
\end{proposition}

\begin{proof} Relation (i) follows from
$$g_{\max\{\phi,\psi\}}\ge\max\{g_\phi,g_\psi\}\ge\max\{\phi,\psi\},$$
the first inequality here being by Proposition~\ref{gphi}(iv).
For (ii), assume first $g_\psi\le g_\phi$, then
$$g_\phi\le g_{\max\{\phi,\psi\}} = g_{\max\{g_\phi,g_\psi\}}= g_\phi.$$
To prove the reverse, we can assume, by (i), $\phi=g_\phi$ and $\psi=g_\psi$. Then we have
$$ \max\{\phi,\psi\}\le g_{\max\{\phi,\psi\}}=\phi,$$
and the proof is complete.
\end{proof}

\begin{proposition}\label{prG&P} Let $\phi,\psi\in \PSH_s^-(D)$. Then
\begin{enumerate}
\item[(i)] $ g_{P(g_\phi,g_\psi)}= P(g_\phi, g_\psi)$;
\item[(ii)] $P(g_\phi, g_\psi)$ is a maximal psh function outside its unboudedness locus;
\item[(iii)] If $g_\phi$ and $g_\psi$ are mutually prime, then $g_{\phi+\psi}=g_\phi+g_\psi$.
\end{enumerate}
\end{proposition}

\begin{proof} By Proposition~\ref{gphi}(v), we get, since $g_\phi$ and $g_\psi$ are idempotent,
   $$ P(g_\phi,g_\psi)\le g_{P(g_\phi,g_\psi)}\le P(g_\phi,g_\psi)$$
     which implies (i) and, in view of Theorem~\ref{gphi2}(i), statement (ii).
  Finally, if $g_\phi$ and $g_\psi$ are mutually prime, then
$$ g_{\phi+\psi} =g_{g_\phi+g_\psi} = g_{P(g_\phi,g_\psi)}= P(g_\phi, g_\psi)= g_\phi+g_\psi$$
by (i) and Proposition~\ref{gphisum}(i).
\end{proof}

\begin{remark} We do not know if $ g_{P(\phi,\psi)}= P(g_\phi, g_\psi)$.
\end{remark}

\section{Green-Poisson vs. Green and Poisson}

In Example~\ref{pcgf}, we had the Green-Poisson residual functions (defined by the global conditions $w\le \phi+C$) equal to the Green functions constructed by the local condition $w\le \phi+O(1)$ near $L(\phi)$ because, in those examples,  $BL(\phi)=\emptyset$. The global conditions take care of the boundary singularities as well and are equivalent to  $w\le \phi+O(1)$ near $L(\phi)\cup BL(\phi)$.
Sometimes it is however useful to focus only on singularities inside $D$ or only near its boundary.  Let us consider
\beq\label{sphio}\cS_\phi^o =  \cS_{\phi,D}^o= \{w\in\PSH^-(D):\: w\le \phi+ O(1) \ {\rm locally\ in\ } D \}, \eeq
\beq\label{sphib}\cS_\phi^b =  \cS_{\phi,D}^b= \{w\in\PSH^-(D):\: w\le \phi+ O(1) \ {\rm locally\ near\ }\partial D \}, \eeq
and the corresponding extremal functions: the {\it residual Green function} $$g_\phi^o={\sup}^*\{w\in\cS_\phi^o\}$$ and the {\it residual Poisson function} $$g_\phi^b={\sup}^*\{w\in\cS_\phi^b\}.$$
Of course, the relation $w\le \phi+ O(1) $ should be controlled only near $L(\phi)$ in (\ref{sphio}), and near $BL(\phi)$ in (\ref{sphib}).

Evidently,
\beq\label{grin}  g_\phi\le P(g_\phi^o,g_\phi^b)\quad \forall\phi\in\PSH^-(D).\eeq
We will see later that, for functions in the Cegrell class $\E$ and small unbounded locus, there is equality here.

Similarly to Proposition~\ref{gphi} and Theorem~\ref{gphi2}, we have the corresponding properties of the residual Green and Poisson functions.

\begin{proposition}\label{tgphi} Let $\phi,\psi\in \PSH^-(D)$. Then
\begin{enumerate}
\item[(i)] $ g_{c\,\phi}^o=c\, g_\phi^o$ for any $c>0$;
\item[(ii)] if $\phi\le\psi+O(1)$ near $L(\phi)$, then $ g_\phi^o\le  g_\psi^o$;
\item[(iii)] $ g_\phi^o$ is a maximal psh function outside $L(\phi)$.
\end{enumerate}

Furthermore, if $\phi\in \PSH_s^-(D)$, then
\begin{enumerate}
\item[(iv)] $ g_\phi^o$ is a maximal psh function outside its unboudedness locus $L(g_\phi^o)$;
\item[(v)]  $g_\phi(z)^o\to0$ when $z\to \partial D \setminus \overline{L(g_\phi^o)}$;
\item[(vi)] $g_{g_\phi^o}^o= g_\phi^o$.
 \end{enumerate}
\end{proposition}

\begin{proposition}\label{tgphi_b} Let $\phi,\psi\in \PSH^-(D)$. Then
\begin{enumerate}
\item[(i)] $ g_{c\,\phi}^b=c\, g_\phi^b$ for any $c>0$;
\item[(ii)] if $\phi\le\psi+O(1)$ near $BL(\phi)$, then $ g_\phi^b\le  g_\psi^b$;
\item[(iii)] $ g_\phi^b$ is a maximal psh function in $D$.
\end{enumerate}

Furthermore, if $\phi\in \PSH_s^-(D)$, then
\begin{enumerate}
\item[(iv)]  $g_\phi(z)^b\to0$ when $z\to \partial D \setminus {BL(g_\phi^b)}$;
\item[(v)] $g_{g_\phi^b}^b= g_\phi^b$.
 \end{enumerate}
\end{proposition}

In addition, we have

\begin{proposition} \label{separate}
If $\phi\in\PSH_s^-(D)$, then
\begin{enumerate}
\item[(i)] $g_{g_\phi^o}=g_{g_\phi}^o=g_\phi^o$;
\item[(ii)] $g_{g_\phi^b}=g_{g_\phi}^b=g_\phi^b$;
\end{enumerate}
\end{proposition}

\begin{proof} We have $g_\phi^o\le g_{g_\phi^o}$ and $g_\phi^o\le g_{g_\phi}^o$. On the other hand, $g_{g_\phi^o}\le g_{g_\phi^o}^o=g_\phi^o$ and $g_{g_\phi}^o\le g_{g_\phi^o}^o=g_\phi^o$ by Proposition~\ref{tgphi}(v), which gives us (i).
Assertion (ii) is proved similarly.
\end{proof}

\medskip

It is not surprising that, in the definition of the class $\cS_\phi$, the relation $w\le \phi+ O(1)$ near $L(\phi)$ is not essential if $\phi$ is a maximal psh function in $D$.

\begin{proposition}\label{gbmax}
If $H\in\PSH^-(D)$ is maximal, then $g_H=g_H^b$.
\end{proposition}

\begin{proof} Evidently, $g_H\le g_H^b$. On the other hand, let $v\in \cS_H^b$, then $v\le H+C_1$ for some $C_1>0$ in $\omega_1=D\cap U_1$ for a neighbourhood $U_1$ of $BL(H)$. In addition, $H\ge -C_2$ for some $C_2>0$ in $\omega_2=D\cap U_2$ for a neighbourhood $U_2$ of $\partial D\setminus U_1$, so $v\le H+C_2$ there. Set $C=\max C_k$. Since $H$ is maximal in $D$ and $\{v> H+C\}\Subset D$, we get $v\le H+C$ in the whole $D$, so $v\in \cS_H$.
\end{proof}

\medskip
Here is an instructive example of functions with purely boundary singularities.

\begin{example}\label{Poisson} {\rm Let
$\Omega\in\DSH_s^-(\D) $ be the negative Poisson kernel for the unit disk $\D$:
$$\Omega(z)=\Omega_{\zeta,\D}(z)=\frac{|z|^2-1}{|1- z\bar\zeta|^2}, \quad \zeta\in\partial\D.$$
Since $\Omega\in L_{loc}^\infty(\D)$, we have $ g_\Omega^o= 0$. On the other hand, since $-\Omega$ is a minimal positive harmonic function in the sense of Martin, the best harmonic majorant of any $w\in \cS_\Omega$ is dominated by $\Omega$, so
$ g_\Omega=g_\Omega^b=\Omega<0$.

A similar effect holds for $\Omega\in\PSH_s^-(\Bn)$ defined by
$$\Omega(z)=\Omega_{\zeta,\Bn}(z)=\frac{|z|^2-1}{|1-\langle z,\zeta\rangle|^2}, \quad \zeta\in\partial\Bn,$$
the harmonicity being replaced by maximality. Note that the restriction of  $\Omega_{\zeta,\Bn}(z)$ to the complex line $z=\lambda\zeta$ equals $\Omega_{1,\D}(\lambda)$.

Furthermore, it was shown in \cite{BPT}, \cite{HW}, \cite{BST} that if $D\Subset\Cn$ is a strongly pseudoconvex domain with smooth boundary, then for any $\zeta\in\partial D$ there exists the {\it pluricomplex Poisson kernel} $\Omega_\zeta\in\PSH^-(D)$, which is maximal psh in $D$, continuous in ${\overline D}\setminus\{\zeta\}$, equal to $0$ on ${\partial D}\setminus\{\zeta\}$, and such that $\Omega_\zeta(z) \approx - |z-\zeta|^{-1}$ as $z\to \zeta$ nontangentially. We then have, by \cite[Prop. 7.1]{BST}, $P(\Omega_\zeta+C,0)=\Omega_\zeta$ for any $C>0$ and so,  $g_{\Omega_\zeta}=g_{\Omega_\zeta}^b=\Omega_\zeta$.
}
\end{example}

\begin{remark}{\rm When $n=1$, a boundary singularity of $\phi\in\SH^-(\D)$ must be strong enough in order to survive the transition to $g_\phi$. Namely, if $BL(\phi)=\zeta\in\partial\D$, then $g_\phi^b=c\,\Omega$ for some $c\ge 0$. So, if $|z-\zeta|\phi(z)\to 0$ for such a $\phi$ as $z\to\zeta$, then $g_\phi^b=0$.
We do not know if this extends to the case $n>1$.}
\end{remark}

In the cases presented in Example~\ref{pcgf}, we have $g_\phi=g_\phi^o $ and $g_\phi^b=0$, while the latter property is far from being true when $L(\phi)$ is not compactly supported inside the domain.

\begin{example}\label{Gball} {\rm One can compare the introduced extremal functions with the Green functions with singularities along complex spaces, introduced in \cite{RSig2}.

1. When $\phi=\log|f|$ for a holomorphic map $f$ to $\C^N$, the function $ g_\phi^o$ coincides, by definition, with the Green function $G_\I$ for the ideal $\I$ generated by the components of $f$, see \cite{RSig2}.

2. In particular, let $\phi_1(z)=\log|z_1|$, then the Green function $ g_{\phi_1,\D^n}^o=G_{\I_1,\D^n}$ for $\I_1=\langle z_1\rangle$ in the unit polydisk $\D^n$ equals $\phi_1\in \cS_{\phi_1,\D^n} $ \cite{LS} and thus coincides with the Green-Poisson function $g_{\phi_1,\D^n}$. Note that it is not equal to zero on the whole $\partial \D^n\setminus \{z_1=0\}$, which does not, however, contradict Proposition~\ref{tgphi}(v) because polydisks are not B-regular domains.

3. For the unit ball $\Bn$,
$$ g_{\phi_1,\Bn}^o=G_{\I_1,\Bn}=\log\frac{|z_1|}{\sqrt{1-|z'|^2}},$$
see \cite{LS}, so $g_{\phi_1,\Bn}\not\in\cS_{\phi,\Bn}$. Note, however,  that $ g_{\phi_1,\Bn}^o=g_{\phi_1,\Bn}$, because the former is the upper envelope of the restrictions of the functions
$$ g_{\phi_1,\rho\Bn}^o(z)= \log\frac{|z_1|}{\sqrt{\rho^2-|z'|^2}}\in \cS_{\phi, \Bn}, \quad \rho>1,
$$
to $\Bn$. In this case, we have $ g_{\phi_1,\Bn}^b=g_{\phi_1,\Bn}$ as well since the condition $w\le \log|z_1| +O(1)$ near $\partial\Bn$ propagates inside $\Bn$, which follows essentially from Siu's semicontinuity theorem.

4. More generally, if the map $f$ considered in the first example above is holomorphic in a neighbourhood of $\overline D$, then the relation $w\le \log|f| +O(1)$ near $\partial D$ propagates inside $D$ as well \cite{R09b} and so, $g_\phi=g_\phi^b$. In addition, if the Green functions $G_{\I,D_k}\searrow G_{\I,D}=g_\phi^o$ for domains $D_k\Supset D$ decreasing to $D$, then $g_\phi\ge g_\phi^o$ because $g_\phi\ge G_{\I,D_k}$ for any $k$ and so, $g_\phi = g_\phi^o$ as well. }
\end{example}

\begin{remark}{\rm Even more generally, we believe $g_{\phi,D}=g_{\phi,D}^o$, provided $\phi\in \PSH(D')$ with $D\Subset D'$. Later on, in Corollary~\ref{geD1}, we will prove this for $\phi$ in the Cegrell class $\E(D')$.}
\end{remark}

%\begin{proposition} Let $D_{-1}\Subset\ldots \Subset D_{-k}\Subset\ldots\Subset D\Subset\ldots\Subset D_{k}
%\Subset\ldots \Subset D_1$ be two families of B-regular domains converging to $D$. Then  $ g_{\phi,D}=G_{\phi,D}$, provided at least one of the following conditions holds:
%\begin{enumerate}
%\item[(i)] the Green-Poisson functions
%$ g_{\phi,D_{-k}}$ converge to $g_{\phi,D}$;
%\item[(ii)]
% $\phi\in\PSH^-(\overline D)$ and the Green functions $ G_{\phi,D_k}$ converge q.e. to $ G_{\phi,D}$. {\bf Isn't it always true?}
%\end{enumerate}
%\end{proposition}
%
%\begin{proof} (i) The restriction of any $v\in\cS_{\phi, D}^*$ to $D_{-k}$ belongs to $\cS_{\phi, D_{-k}}$, so $ g_{\phi,D_{-k}}\ge G_{\phi,D}$. By the assumption, $ %g_{\phi,D_{-k}}$ decrease to $ g_{\phi,D}\le G_{\phi,D}$ as $k\to +\infty$, and the assertion follows.
%
%(ii) is proved similarly.
%\end{proof}

\section{Model and approximately model singularities}
The class $\cS_\phi$ need not be closed, so  $g_\phi$ can have weaker singularity than $\phi$. For example, if $BL(\phi)=\emptyset$ and $(dd^c \phi)^n$  does not charge any pluripolar set, then $g_\phi=0$, see \cite{R06}.

We say, as for quasi-psh functions on compact manifolds \cite{DNL18a}, that $\phi$ has {\it model singularity} in $D$ if $\phi\sim_D g_\phi$, i.e., $g_\phi\le\phi+C$  in $D$, in which case $g_\phi\in\cS_\phi$.
For instance, the function $\log|z-a|$ has model singularity in any $D$ containing $a$ and, more generally, so does any  $\phi$ which is maximal and bounded outside a finite set $A\Subset D$. Another example of model singularity is $\phi=\log|F|$ for a holomorphic map $F:D\to{\mathbb B}^N$, $N\ge n$, with $|F|\ge\delta>0$ near $\partial {\mathbb B}^N$ \cite{RSig2}. Functions with isolated, but non-model singularities (both interior and boundary), are presented in Examples~\ref{exaas}.2 -- \ref{exaas}.4.

On the other hand, the singularity of the function $\log|z_1|$ is not model in $\Bn$, see Example~\ref{Gball}.3. More generally, a typical function $\phi$ with $L(\phi)\not\Subset D$ does not have model singularity. Indeed, let us assume $u\in\PSH^-(D)$ can be extended to a larger domain $D'$. If $g_\phi(\zeta_0)=\phi(\zeta_0)=-\infty $ for some $\zeta_0\in D'\cap \partial D$, then $g_\phi(z_k)-\phi(z_k)\to +\infty$ when $z_k\to\zeta_0$ inside $\{g_\phi(z)>-1\}$.

This makes us want to relax the requirements to similarity of the singularities of $\phi$ and $g_\phi$.
We say that $\phi$ has {\it approximately model  singularity} in $D$ if $g_\phi\preceq_D \sigma \phi$ for any $\sigma\in (0,1)$, which is equivalent to the condition
$$ \frac{g_\phi(z)}{\phi(z)}\to 1 \quad {\rm as\ } \phi(z)\to -\infty.$$
This fixes Example~\ref{exaas}.2, however it will not work in some natural situations as in Example~\ref{Gball}.3, where the ratio equals $1/2$ for all $z\in\Bn$ with $|z_1|=\epsilon$ and $|z'|=\sqrt{1-\epsilon}$.

Having this example in mind, we will say that $\phi\in\PSH^-(D)$ has {\it locally model} or {\it locally approximately model} singularity if
if its restriction to any $D'\Subset D$ has model (resp., approximately model) singularity with respect to $g_{\phi,D'}$; since $g_\phi\le g_{\phi,D'}$, we then have $g_\phi\preceq_{D'} \phi$ or $g_\phi\preceq_{D'} \sigma\phi$, respectively. Evidently, having locally approximately model singularity is equivalent to the condition
$$ \frac{g_\phi(z)}{\phi(z)}\to 1 \quad {\rm as\ } \phi(z)\to -\infty,$$
and the local variant implies global in the absence of boundary singularities.

Here we present some instances when such singularities appear naturally.

A function $\phi\in\PSH^-(D)$ is said to have {\it analytic singularity} in $D$ if there exist $c> 0$ and a bounded holomorphic mapping $f: D\to\C^N$ such that
$\phi\sim c\log|f|$; if this is the case, then $g_\phi= c\,g_{\log|f|}$.
As mentioned above, a noncompact analytic singularity cannot be model. On the other hand, by \cite[Thm. 2.8]{RSig2}, every analytic singularity $\phi$ is locally model; then it will be model if $BL(\phi)=\emptyset$ or, more generally, if $\phi$ extends to a function with analytic singularity in $D'\Supset D$.

Sums and maxima of functions with analytic singularities need not have analytic singularities. This is one of motivations to consider a wider class introduced in \cite{R13a}.
A function $\phi\in\PSH^-(D)$ is said to have {\it asymptotically analytic singularity} if for any $\epsilon>0$ there exists $\phi_\epsilon$ with analytic singularity such that $(1+\epsilon)\phi_\epsilon(z) \preceq \phi \preceq (1-\epsilon)\phi_\epsilon(z)$, that is,
\beq\label{aas}(1+\epsilon)\phi_\epsilon(z)-C_\epsilon \le \phi \le (1-\epsilon)\phi_\epsilon(z)+C_\epsilon  \eeq
for some $C_\epsilon>0$.
Note that $L(\phi)=L(\phi_\epsilon)$ for any $\epsilon>0$, so it is a closed analytic variety.

\begin{example}\label{exaas} {\rm 1. Simple examples of asymptotically analytic, but not analytic, singularities are given by  $\log(|f_1|+|f_2|^\gamma)$ and $\log|f_1|+\gamma\log|f_2|$ for irrational $\gamma>0$.

2. The function $\phi=\log|z|-|\log|z||^{1/2}$  has asymptotically analytic singularity in any $D\Subset \Bn$ because it satisfies there (\ref{aas}) with $\phi_\epsilon =\log|z|$ for any $\epsilon>0$. Since $g_\phi=\log|z|+O(1)$, the singularity is not model, but it is still approximately model.

3. On the other hand, the function $\psi=\max\{\log|z_1|,-|\log|z_2||^{1/2}\}$ does not have asymptotically analytic singularity near $0$, which can easily be checked by considering its restriction to the line $\{z_1=0\}$. Moreover, the singularity is not even approximately model because $g_{\psi}\ge g_{\phi_\epsilon}\ge \phi_\epsilon:= \max\{\log|z_1|,\epsilon\log|z_2|\}\to 0$ outside $\{z_1=0\}$ as $\epsilon \to 0$, so $g_\psi=0$. To get a (toric) psh function $\phi$ with $g_\phi\neq 0$ which is not approximately model, take $\phi=\log|z|+\psi$, then, by Proposition~\ref{gphisum}(ii), $g_\phi=g_{\log|z|}=\log|z| + O(1)$ in any $D\subset\Bn$ containing $0$, and
$$\limsup_{z\to 0}\frac{\phi(z)}{g_\phi(z)}\ge 2$$ because $\phi(z)= 2\log|z_1|$ on $\{z_2=0\}$.

4. If $\phi(z)=\log|z-a|-C $ with $|a|=1$ and $C>>0$, then its Green-Poisson function $g_\phi$ for $\Bn$ is identical zero. For $n=1$, this is because the Poisson kernel with pole at $a$ is a minimal positive harmonic function in the unit disk, while for $n>1$ this follows from a more general fact given in Corollary~\ref{geD1} in the next section. Therefore, the singularity of $\phi$ is model in any $\B_r^n$ with $r\neq 1$, while it is not even approximately model in $\Bn$.
}
\end{example}

\begin{proposition}\label{summaxaas} If $\phi$ and $\psi$ have asymptotically analytic singularities, then so do $\phi+\psi$, $\max\{\phi,\psi\}$, and $P(\phi,\psi)$.
\end{proposition}

\begin{proof} It suffices to proof this for $\phi= a\log|f|$ and $\psi=b\log|h|$ with $f,h$ holomorphic mappings to $\C^M$ and $\C^N$. If $a/b$ is rational, say, $a=\frac{p}{q}b$, then
$\phi+\psi= \frac{b}{q}\log|F|$, where $F$ is the mapping to $\C^{MN}$ with components $f_j^ph_k^q$, $1\le j,k\le N$. Otherwise, for any $\epsilon\in(0,1)$ one can find $a_\epsilon,b_\epsilon\in\Q_+$ such that
$$(1+\epsilon)a_\epsilon \log|f|\le \phi \le (1-\epsilon)a_\epsilon \log|f|$$
and
$$(1+\epsilon)b_\epsilon \log|h|\le \psi \le (1-\epsilon)b_\epsilon \log|h|,$$
so
$$(1+\epsilon)c_\epsilon \log|F_\epsilon|\le \phi \le (1-\epsilon)c_\epsilon \log|F_\epsilon|$$
with $a_\epsilon=\frac{p_\epsilon}{q_\epsilon}b$, $c_\epsilon=b_\epsilon/q_\epsilon$, and $F_\epsilon$ the mapping with components $f_j^{p_\epsilon}h_k^{q_\epsilon}$.

For $\max\{\phi,\psi\}$ and $P(\phi,\psi)$, the proofs are similar. In the former case, it is done by using the mapping to $\C^{M+N}$ whith components $f_j^p$ and $h_k^q$, $1\le j,k\le N$, and in the latter one, the mapping whose components are least common multiples of $f_j^p$ and $h_k^q$, see Proposition~\ref{analenv}.
\end{proof}

\medskip

An asymptotically analytic singularity $\phi$ need not be locally model; the Green-Poisson function $g_\phi$ for $\phi$ from Example~{\ref{exaas}.2 in the ball $\B_r^n$, $r<1$,  equals $\log|z|-\log r$ and so, $\phi-g_\phi$ is unbounded in any neighbourhood of $0$.
On the other hand, it is not hard to see that asymptotically analytic singularities are locally approximately model; for the case of finite $L(\phi)$, this was mentioned in \cite{R13a}.

\begin{proposition}\label{aaslim} Any asymptotically analytic singularity $\phi$ is locally approximately model.
\end{proposition}

\begin{proof}
Since ${g_{\phi}(z)}/{\phi(z)}\le 1$, we need to estimate the fraction from below near any point $a\in L(\phi)$. The second inequality in (\ref{aas}) implies $g_{\phi}\le (1-\epsilon)g_{\phi_\epsilon}$. Furthermore, since $\phi_\epsilon$ has analytic singularity, $g_{\phi_\epsilon}\le \phi_\epsilon+A_\epsilon$ near $a$ for some $A_\epsilon>0$. Combining this with the first inequality in (\ref{aas}), we get
$$\frac{g_{\phi}(z)}{\phi(z)}\ge \frac{(1-\epsilon)[\phi_\epsilon(z) +A_\epsilon]}{(1+\epsilon)\phi_\epsilon(z)-C_\epsilon}\rightarrow
 \frac{1-\epsilon}{1+\epsilon}
$$
as $z\to a$, which proves the statement.
\end{proof}

\medskip
One can also see that the standard operations on asymptotically analytic singularities are well coordinated with their Green-Poisson functions.

\begin{proposition}\label{grsum_an} If $\phi,\psi\in\PSH^-(D)$  have asymptotically analytic singularities, then
$$\lim_{z\to a}\frac{g_\phi(z)+g_\psi(z)}{ g_{\phi+\psi}(z)}=1, \quad a\in L(g_\phi)\cup L(g_\psi)$$
$$\lim_{z\to a}\frac{g_{\max\{\phi,\psi\}}}{ \max\{g_\phi(z),g_\psi(z)\}}=1, \quad a\in L(g_\phi)\cap L(g_\psi),$$
and
$$\lim_{z\to a}\frac{P(g_\phi,g_\psi)(z)}{ g_{P(\phi,\psi)}(z)}=1, \quad a\in L(g_\phi)\cup L(g_\psi).$$
\end{proposition}

\begin{proof} This follows from Propositions~\ref{summaxaas}, \ref{aaslim}, \ref{gphisum}(i), \ref{gmax}(i), and \ref{prG&P}(i).
\end{proof}

%%%%%%%%%%%%%%%%%%%%

\section{Cegrell classes}\label{sec:Ceg}

Here we specify the notions considered in the previous sections for functions of the Cegrell class $\E$, the largest class of negative psh functions $\phi$ for which the Monge-Amp\'ere operator is well defined \cite{Ce04}, \cite{Bl}; in particular, it is continuous on $\E$ with respect to monotone convergence (both decreasing and increasing) and to convergence in capacity. Furthermore, $\phi\in\E(D)$ belongs to $\F(D)$ if and only if $\int_D (dd^c\phi)^n<\infty$ and the least maximal psh majorant of $\phi$ in $D$ is the identical zero \cite{Ce08}.
In what follows, we will use a machinery developed in \cite{ACCP}, \cite{Ce04}--\cite{Ce12}.

First we note that, by \cite{BT87} (see also \cite[Thm. 5.11]{Ce04}), for any function $\phi\in\E(D)$ and any pluripolar set $A$,
$(dd^c \phi)^n(A)=(dd^c \phi)^n(A\cap L'(\phi))$, where
$$L'(\phi)=\{\phi(z)=-\infty\}\subset L(\phi).$$
This and  Proposition~\ref{NPMAg} imply that the measure
$(dd^c g_\phi)^n$  vanishes outside $L'(g_\phi)$, without assuming $\phi$ to have small unboundedness locus. In other words:
\beq\label{cargphi} (dd^c g_\phi)^n=\Bone_{L'(g_\phi)} (dd^c g_\phi)^n.\eeq
It is worth mentioning that functions from $\E$ need not have small unboundedness locus; for example, there exists $\phi\in\E(D)$ such that $L(\phi)=D$ \cite[Ex. 2.1]{ACP}; on the other hand, there exists $\phi\not\in\E(D)$ such that $L'(\phi)=\emptyset$ \cite[Ex. 4.6]{ACP}. Note also that $L'(g_\phi)$ need not coincide with $L(g_\phi)$.

\begin{example}   {\rm
Let $D$ be the unit disk $\D$ of $\C$ and $\phi(z)=\sum_k a_k\log\frac{|z-w_k|}2$ for a sequence $w_k\in\D\setminus\{0\}$ converging to $0$ and $a_k=-2^{-k}/\log|w_k/2|$. Then
$g_\phi=\sum_k a_k G(z,w_k)$, where $G(z,w)$ is the usual Green function with pole at $w$, so $L'(g_\phi)=W=\cup_k w_k$ while $L(g_\phi)=W\cup\{0\}$.
}
\end{example}

We will repeatedly use the following version of Demailly's {\sl Comparison Theorem}.

\begin{theorem}\label{CT} {\rm \cite[Lem. 4.1]{ACCP}}
If $u,v\in\E(D)$ are such that $u\le v$, then $\Bone_A (dd^cu)^n\ge \Bone_A (dd^cv)^n$ for any pluripolar Borel set $A$.
\end{theorem}

As its first application, we show that  $(dd^c g_\phi)^n$ is the residual part of $(dd^c\phi)^n$.

\begin{proposition}\label{gphiphi} If $\phi\in\E(D)$, then
\beq\label{resgphi} (dd^c g_\phi)^n = \Bone_{L'(\phi)}(dd^c\phi)^n = \Bone_{L'(g_\phi)}(dd^c\phi)^n.
\eeq
\end{proposition}

\begin{proof} Take $\phi_j=P(\phi+j,0)$ increasing q.e. to $g_\phi$. Since $(dd^c\phi_j)^n\to (dd^cg_\phi)^n$, we have
$$\limsup_{j\to\infty}\Bone_{L'(\phi)}(dd^c\phi_j)^n \le \lim_{j\to\infty}(dd^c\phi_j)^n = (dd^cg_\phi)^n,$$
while, by Theorem~\ref{CT} and (\ref{cargphi}),
$$\Bone_{L'(\phi)}(dd^c\phi)^n=\Bone_{L'(\phi)}(dd^c\phi_j)^n\ge \Bone_{L'(\phi)}(dd^c g_\phi)^n=(dd^c g_\phi)^n$$
for any $j$, and (\ref{resgphi}) follows.
\end{proof}

\begin{corollary}\label{ggphi} If $\phi\in\E(D)$, then
$ (dd^c g_{g_\phi})^n=(dd^cg_\phi)^n$.
\end{corollary}

\begin{proof}
Applying Proposition~\ref{gphiphi} to $g_\phi$ instead of $\phi$, we get
$$ (dd^c g_{g_\phi})^n=\Bone_{L'(g_\phi)} (dd^c g_\phi)^n= (dd^cg_\phi)^n.$$
\end{proof}

\medskip

The easiest is the case of functions from the Cegrell class $\F$, where the following version of {\sl Identity Principle} holds: the conditions $u\le v$ and $(dd^cu)^n=(dd^cv)^n$
imply $u=v$ \cite{NPh} (later on, in Theorem~\ref{IP}, we recall a more general result). Using this, we get

\begin{proposition}\label{F}
If $\phi\in\F(D)$, then $g_\phi\in\F(D)$ and $g_\phi=g_{g_\phi}$.
In particular, $g_\phi=0$, provided $(dd^c \phi)^n$ does not charge pluripolar sets.
\end{proposition}

\begin{proof} Since $g_{g_\phi}\ge g_\phi\ge \phi$, both the Green-Poisson functions belong to $\F(D)$ as well \cite{Ce04}.  By Corollary~\ref{ggphi} and the aforementioned Identity Principle,  the two functions coincide.

If $(dd^c g_\phi)^n$ does not charge pluripolar sets, then $(dd^c g_\phi)^n=0$, and the only function in $\F(D)$ with this property is $0$ (this is again by the Identity Principle).
\end{proof}

\medskip

To work with larger classes of functions than $\F$, one can use a machinery of boundary values, developed in \cite{Ce08} and \cite{ACCP}.
Given $\phi\in\E(D)$, let $\fb_D\phi$ be its best maximal psh majorant in $D$, that is, the least maximal psh function in $D$ greater or equal to $\phi$; when the domain $D$ is clear from the context, we will simply write $\fb\phi$. (Our denotation differs from that in  \cite{Ce08} and other papers on the subject, where $\tilde \phi$ is used.) In \cite{Ce08}, it was constructed as the regularized limit of the functions
\beq\label{fund}\phi_j=\sup\{w\in\PSH^-(D):\: w\le \phi \ {\rm on } \ D\setminus D_j\}\eeq
for a sequence of strictly pseudoconvex domains $D_j\Subset D_{j+1}\Subset\ldots\Subset D$ such that $\cup_j D_j=D$. Since $\phi_j$ satisfy $(dd^c\phi_j)^n=0$ on $D_j$ and increase q.e. to $\fb\phi\in\E(D)$, the latter is maximal in $D$. Note that  $\phi\in\E(D)$ is maximal if and only if $\fb\phi=\phi$.

Denote
$$\N(D)=\{w\in\E(D):\: \fb w=0\}.$$
Furthermore, given a function $H\in\E(D)$, let
$\N(D,H)$ and $\F(D,H)$ denote the classes of all functions $\phi\in\E(D)$ such that
\beq\label{bv} H+w\le \phi\le H\eeq
for some $w\in\N(D)$ or $w\in\F(D)$, respectively. Since $\F(D)\subset\N(D)$, we have $\F(D,H)\subset\N(D,H)$.
When $\phi\in\N(D,H)$ for a maximal psh function $H\in\E(D)$, then $H$ equals the least maximal majorant $\fb\phi$ of $\phi$, so the relation $\phi\in\N(D,\fb\phi)$ means that $\phi$ and $\fb\phi$ have the same boundary values in the sense of (\ref{bv}). In such a case, $\fb\phi$ will be referred to as {\it the boundary value} of $\phi$. In particular, this is so if  $\phi$ has finite total Monge-Amp\`ere mass: by  \cite[Thm. 2.1]{Ce08},
there exists $w\in\F(D)$ with
$(dd^c w)^n(D)\le (dd^c\phi)^n(D)$ such that $\fb\phi+w\le \phi\le \fb\phi$.

\medskip

The importance of this approach is clear from the following {\sl Identity Principle}:

\begin{theorem}\label{IP} {\rm \cite[Thm. 3.6]{ACCP}}
If $u,v\in\N(D,H)$ are such that $u\le v$, $(dd^c u)^n = (dd^cv)^n$ and $\int_D(-w)(dd^c u)^n<\infty$ for some non-zero $w\in\E(D)$,  then $u=v$.
\end{theorem}

Boundary values of $g_\phi$ are described by

\begin{theorem}\label{bd} Let $\K\in \{\N,\F\}$.
\begin{enumerate}
\item[(i)]
If $\phi\in\K(D,\fb\phi) $, then
\beq\label{gphib}\fb g_\phi=g_{\fb\phi}= g_\phi^b,\eeq
$g_\phi\in\K(D, g_{\fb\phi})$, and $g_\phi^o\in\K(D, g_{\fb\phi}^o)$; in particular, $ g_\phi^o\in\K(D)$ if $\fb\phi\in L_{loc}^\infty(D)$.
\item[(ii)] If, in addition, $g_{\fb\phi}$ is idempotent (for example, if $\fb\phi\in\E(D)\cap\PSH_s(D)$), then
$g_\phi$ is idempotent as well and
\beq\label{NMphi} g_{\fb\phi}+ g_\phi^o\le g_\phi\le g_{\fb\phi}.\eeq
\end{enumerate}
\end{theorem}

\begin{proof} The boundary value of $g_\phi$, if exists, equals $\fb g_\phi$. On the other hand,
the condition $\phi\in\K(D,b\phi) $ means that there exists $w\in\K(D)$ such that
 \beq\label{mphi} \fb\phi+w\le \phi\le \fb\phi.\eeq
 By Proposition~\ref{gphi}(iii), we then have
\beq\label{mphi1} g_{\fb\phi}+g_{w}\le g_{\fb\phi+w}\le g_\phi\le g_{\fb\phi},\eeq
which proves $g_\phi\in\K(D, g_{\fb\phi})$ because $g_w\in\K(D)$.
This implies $\fb g_\phi=g_{\fb\phi}$.

The function $\fb g_\phi$ is the least maximal majorant of $g_\phi$, while $g_\phi^b$ is one of its maximal majorants, so $g_\phi^b\ge \fb g_\phi$.
On the other hand, by Proposition~\ref{gbmax}, $g_{\fb\phi}=g_{\fb\phi}^b$ and, since $g_{\fb\phi}^b\ge g_{\phi}^b$, we get the second equality in (\ref{gphib}).

Similarly, (\ref{mphi}) implies
$$  g_{\fb\phi}^o+ g_{w}^o\le  g_\phi^o\le  g_{\fb\phi}^o,$$
and we get $ g_\phi^o\in\K(D,g_{\fb\phi}^o)$. In particular, $ g_\phi^o\in\K(D)$, provided $ g_{\fb\phi}^o=0$ (which is the case if, for instance, $\fb\phi\in L_{loc}^\infty(D)$).

Assuming, in addition, $g_{g_{\fb\phi}}=g_{\fb\phi}$, we get by (\ref{mphi1}),
$$ g_{\fb\phi}+ g_v\le g_{g_{\fb\phi}+v} \le g_{g_\phi}\le g_{\fb\phi} $$
with $v=g_w\in\K(D)$,
which means that $g_{g_\phi}\in\K(D, g_{\fb\phi})$ as well. Since, by Corollary~\ref{ggphi}, $(dd^cg_\phi)^n=(dd^cg_{g_\phi})^n$, and $g_\phi\le g_{g_\phi}$, Theorem~\ref{IP} establishes the equality.

Finally, to prove  (\ref{NMphi}), let $\phi_j$ be defined as in (\ref{fund}), then $\phi_j\le 0$ in $D_j$ and $\phi_j\le \phi$ in $D\setminus D_j$. Furthermore, if $\psi\in\cS_\phi^o$, then $\psi\le \phi+C_{\psi,j}$ in $D_j$ for some $C_{\psi,j}\ge0$, while $\psi\le 0$ in $D\setminus D_j$. Take any sequence $\psi_j\in\cS_\phi^o$ increasing q.e. to $ g_\phi^o$, then $\phi_j+\psi_j \in\cS_\phi$.
Letting $j\to\infty$, we get
$$ \fb\phi+ g_\phi^o\le g_\phi.$$
Therefore,
$$ g_{\fb\phi}+ g_\phi^o\le g_{\fb\phi}+ g_{g_\phi^o}\le g_{\fb\phi+ g_\phi^o}\le g_{g_\phi},$$
and we get the first inequality in (\ref{NMphi}), provided $g_{g_\phi}=g_\phi$. Since the second one is obvious, this completes the proof.
\end{proof}

\medskip
Since $g_\phi^b=0$ for any $\phi\in\F(D)$, (\ref{NMphi}) implies

\begin{corollary}  $g_\phi=g_\phi^o$ for any $\phi\in\F(D)$.
\end{corollary}

Let $\E^a(D)$ denote the collection of all functions $\phi$ in $\E(D)$ whose Monge-Amp\`ere measure do not charge pluripolar sets.

\begin{corollary}\label{gea}
If $\phi\in\N(D,\fb\phi)\cap\E^a(D)$,
then $g_\phi=g_\phi^b$.
\end{corollary}

\begin{proof} Since, by Theorem~\ref{bd}(i), $g_\phi\in\N(D,g_\phi^b)$ and, by Proposition~\ref{gphiphi}, it is maximal in $D$, the assertion follows from Theorem~\ref{IP}.
\end{proof}

\begin{corollary}\label{geD1}
If $D\Subset D'$ and $\phi\in\E(D')$, then $g_{\phi,D}^b=0$
and
$ g_{\phi,D}=g_{\phi,D}^o$. In particular, $g_{\phi,D}=0$, provided $\phi\in\E(D')\cap \E^a(D)$.
\end{corollary}

\begin{proof} Note first that, since $\phi$ has finite total MA mass in $D$, it belongs to $\F(D, \fb_D\phi)$ by the already mentioned \cite[Thm. 2.1]{Ce08}.
By the definition of the class $\E(D')$, there exists a function $\psi\in\F(D')$ coinciding with $\phi$ near $\overline {D}$. Then
$$ g_{\phi,D}^b= g_{\fb_D\phi, D}=  g_{\fb_D\psi, D}\ge  g_{\fb_{D'}\psi, D'}=0,
$$
which proves $g_{\phi,D}^b=0$.  The statement on $g_{\phi,D}^o$ follows now from Theorem~\ref{bd}.
\end{proof}

\begin{corollary}\label{gGmeas}
If $\phi\in\N(D,\fb\phi)$ and $\fb\phi\in\PSH_s(D) $,
then $(dd^c g_\phi^o)^n=(dd^c g_\phi)^n$.
\end{corollary}

\begin{proof}
From the first inequality of (\ref{NMphi}), we have
 $ g_{\fb\phi}+ g_\phi^o\le g_\phi\le g^o_\phi$, and Theorem~\ref{CT} gives
 \beq\label{tgG} (dd^c g_\phi^o)^n\le (dd^c g_\phi)^n\le \Bone_{L'(\phi)} \left(dd^c ( g_{\fb\phi}+ g_\phi^o)\right)^n.\eeq
  As follows from \cite[Lem. 4.4]{ACCP},
  \beq\label{Mink}
  \left[\int_A (dd^c(u+v))^n\right]^{1/n}\le \left[\int_A (dd^c u)^n\right]^{1/n} + \left[\int_A (dd^c v)^n\right]^{1/n}
  \eeq
  for any $u,v\in\E(D)$ and any pluripolar $A\subset D$, so
  the maximality of $ g_{\fb\phi}$ implies
 that the right hand side of (\ref{tgG}) equals $(dd^c g_\phi^o)^n$, which completes the proof.
 \end{proof}

\begin{corollary}\label{genv} If $\phi\in\N(D,\fb\phi)$ has finite total residual Monge-Amp\`ere mass and $g_\phi$ is idempotent (for example, if $\fb\phi\in\PSH_s(D)$), then
$$g_\phi= P(g_\phi^o, g_{\phi}^b).$$
 \end{corollary}

\begin{proof} Denote $u=P(g_\phi^o, g_{\phi}^b)$; by (\ref{grin}), $u\ge g_\phi$.

Due to (\ref{MAP}) and (\ref{cargphi}), $(dd^cu)^n=(dd^cg_\phi)^n=0$ outside $L'(g_\phi)$. By Theorem~\ref{CT} and Corollary~\ref{gGmeas}, $(dd^c u)^n\ge(dd^c g_\phi^o)^n= (dd^c g_\phi)^n$; on the other hand, since $g_\phi\le u$, we have, by Theorem~\ref{CT}, $(dd^c u)^n \le (dd^c g_\phi)^n$, which shows that the two measures are equal.
As, by Theorem~\ref{bd}(i), $g_\phi\in\N(D,g_{\phi}^b)$, and $g_\phi\le u\le g_{\phi}^b$, then $u\in\N(D,g_{\phi}^b)$ as well, and Theorem~\ref{IP} implies $g_\phi=u$.
 \end{proof}

\medskip
If $\phi_j$ increase q.e. to $\phi$, then, evidently, $\lim^* g_{\phi_j}\le g_\phi$. For `nice' $\phi_j$, we have a continuity.

\begin{theorem}\label{gincr}
Let $\phi_j\in\E(D)$ with uniformly bounded total MA masses increase q.e. to $\phi\in \E(D)$ with $\fb\phi\in\PSH_s$. Then $g_{\phi_j}\nearrow g_\phi$ and  $g_{\phi_j}^b\nearrow g_\phi^b$ q.e. in $D$. If, in addition, $g_{\phi_j}^o\in\N(D)$, then $g_{\phi_j}^o\nearrow g_\phi^o$ q.e. in $D$.
\end{theorem}

\begin{proof}
We may assume that all $\int_D (dd^c\phi_j)^n\le M$ for some $M>0$ (and thus, the total MA mass of $\phi$ is bounded by $M$ as well), so $\phi_j\in\F(D,\fb\phi_j)$ by the aforementioned \cite[Thm. 2.1]{Ce08}, see the discussion before Theorem~\ref{IP}. More precisely, for each $j$ there exists $w_j\in\F(D)$ such that
$\int_D (dd^c w_j)^n\le M$ and $\fb\phi_j+w_j\le \phi_j\le b\phi_j$. Then we also have
\beq\label{gtilde}g_{\fb\phi_j}+w_j\le g_{\phi_j}\le g_{\fb\phi_j}.\eeq
By \cite[Appendix]{De09},
there exists a subsequence $w_{j_k}$ converging in $L_{loc}^1(D)$ to some $w\in\F(D)$; we can assume $w_j\to w$.

The sequence $\fb\phi_j$ increases q.e. to a maximal psh function $h$, satisfying $\fb\phi\le h\le g_{\fb\phi}$. Since $h\in\PSH_s$, we have $g_{g_{\fb\phi}}=g_{\fb\phi}$, which implies $g_h=g_{\fb\phi}$ and then,
exactly as in the proof of Theorem~\ref{gphi2}(iii), $g_h=h$. Therefore,
$g_{\fb\phi_j}\to g_{\fb\phi}$.
By Theorem~\ref{bd}(i), $g_\psi^b=g_{\fb\psi}$ for any $\psi\in\F(D,\fb\psi)$, so we have apparently proved $g_{\phi_j}^b\to g_\phi^b$.

The functions $g_{\phi_j}$ increase q.e. to $u\in\E(D)$ with $(dd^cu)^n=\Bone_{L'(\phi)}(dd^c\phi)^n$. Then, by (\ref{gtilde}),
$ g_{\fb\phi_j}+w_j\le u\le g_{\fb\phi_j}$
for all $j$, and passing to the limit as $j\to\infty$ we get $g_{\fb\phi}+w\le u\le g_{\fb\phi}$. Therefore, we have $u\in\F(D, g_{\fb\phi})$. By by Theorem~\ref{bd}(i),  $g_\phi\in \F(D, g_{\fb\phi})$ as well.
Since $u\le g_\phi$ and $(dd^cu)^n = (dd^cg_\phi)^n$, Theorem~\ref{IP} implies $u=g_\phi$. This proves $g_{\phi_j}\to g_\phi$.

Similarly, $g_{\phi_j}^o\to v\le g_\phi^o$ with $(dd^cv)^n=(dd^c g_\phi^o)^n$, and the two functions coincide, provided both belong to $\N(D)$.
\end{proof}

\begin{remark} {\rm If $\phi_j$ decrease to $\phi$, then the limit $g_{\phi_j}$ exists but does not need to coincide with $g_\phi$; a simple example is $\phi_j=\max\{\phi,-j\}\in L^\infty(D)$, so  $g_{\phi_j}\equiv 0$ irrespectively of $\phi$.}
\end{remark}

%%%%%%%%%%%%%%%%%%%%%%%%%%%%%%%%%%%%%%%%%%%%%%%%%%%%%%%%%%%%%%%%%%%%%%%%%%%%%%%%%%%%%%%%%%%%

\section{Asymptotic rooftops with respect to singularities}\label{ers}

Let $\phi,\psi\in\PSH^-(D)$.
The function
\beq\label{pvu} P[\phi](\psi)={\sup}^* \{P(\psi,w):\: w\in\PSH^-(D), \ w\sim \phi\}\eeq
is {\it the asymptotic envelope}, or {\it asymptotic rooftop}, of $\psi$ with respect to the singularity of $\phi$}.
Equivalently,
$$  P[\phi](\psi)={\sup}^* \{P(\psi,\phi+C):\: C\in \R\} ={\lim}^*_{C\to\infty} \,P(\psi,\phi+C).$$

It was shown in \cite{R16} that for any $\phi,\psi\in\F_1(D)$, one has $P[\phi](\psi)=\psi$. The argument was based on the proof of the corresponding fact,  \cite[Thm.~4.3]{Da14a}, for $\omega$-psh functions on a compact K\"{a}hler manifold $(X,\omega)$ with full Monge-Amp\`ere mass, and used finiteness of the Monge-Amp\`ere energy in the class $\F_1(D)$.

Here we will see that, actually, no energy consideration is needed at all if $\phi,\psi\in\F_1(D)$  and even more generally, if $\fb \phi$ is not `too wild' at the boundary and $(dd^c\phi )^n$ does not charge pluripolar sets. We will also explore some other situations where the relation $P[\phi](\psi)=\psi$ for all $\psi\le g_\phi$ takes place; note that the latter condition is always fulfilled if $g_\phi=0$, which is the case if $\phi\in\F_1(D)$.

It follows directly from the definition that
$P[\phi](0)=g_\phi$ and
\beq\label{mainin} P[\phi](\psi)\le P(\psi,g_\phi).\eeq
In particular, it gives us

\begin{proposition}\label{onedir} If $P[\phi](\psi)=\psi$, then $\psi\le g_\phi$. If $g_\phi$ is idempotent, this implies ${g_\psi}\le g_\phi$.
\end{proposition}

Furthermore,
we have
\beq\label{mainq} P[\phi](\psi)=P(\psi,g_\phi) \eeq
for any $\psi$, if $\phi$ has model singularity. An intriguing question is if this remains true for any $\phi$ and $\psi$.

\begin{remark}\label{psiless} {\rm
It suffices to establish (\ref{mainq}) for all $\psi\le g_\phi$, in which cases it takes the form
\beq\label{mainq1} P[\phi](\psi)=\psi.\eeq
Indeed, denoting $\psi'=P(\psi,g_\phi)\le g_\phi$
and assuming $P[\phi](\psi')=\psi'$,
we have
$$ P[\phi](\psi)\ge P[\phi](\psi') =  \psi'=P(\psi,g_\phi)$$
which, in view of (\ref{mainin}), gives us (\ref{mainq}).
Obviously, (\ref{mainq1}) is true if $\psi=g_\phi$ or if $\psi\preceq \phi$.}
\end{remark}

\begin{proposition}\label{propyes} Relation
(\ref{mainq}) is true, provided one of the following conditions is fulfilled:
\begin{enumerate}
\item[(i)] $\phi\ge g_\phi+w$ with $w\in\PSH^-(D)$ such that $g_w=0$;
\item[(ii)] $g_\phi=0$;
\item[(iii)] $\phi$ has approximately model singularity.
\end{enumerate}
\end{proposition}

\begin{proof} By (\ref{mainin}), it suffices to establish the inequality
\beq\label{converse} P[\phi](\psi)\ge P(\psi,g_\phi).\eeq
Assuming (i), we have for any $C>0$,
$$ P(\psi,\phi+C)\ge P(\psi,g_\phi+w+C)\ge P(\psi,g_\phi)+P(0, w+C), $$
When $C\to\infty$,   $P(\psi,\phi+C)$ increases q.e. to $P[\phi](\psi)$ and $ P(0, w+C)$  to $g_w=0$, we get (\ref{converse}).
Condition (ii) is a particular case of (i) with $w=\phi$, and in (iii), we have
$$P[g_\phi](\psi)\le P[\sigma\phi](\psi)\le\sigma P[\phi](\psi)$$
for any $\sigma<1$.
\end{proof}

\begin{remark}\label{remyes} {\rm 1. Proposition~\ref{propyes}(ii) implies a result from \cite{R16} for $\phi,\psi\in\F_1(D)$ because, in this case, $g_\phi=0$.

\medskip
2. More generally, if $\phi\in \F(D,\fb\phi)\cap\E^a(D)$ and $g_{\fb\phi}=0$, then $g_\phi=0$ as well and so, $P[\phi](\psi)=P(\psi,g_\phi)$ for any $\psi\in\PSH^-(D)$.

\medskip
3. When $n=1$, the function $\phi-g_\phi$ extends to a negative subharmonic function in $D$, so  $\phi=g_\phi +w$ with $w=\phi-g_\phi\in\DSH(D)$. Moreover, since
$g_\phi=g_{w+g_\phi}= g_w+ g_\phi$, we have $g_w=0$ and the condition in Prop.~\ref{propyes}(i)  is fulfilled and so, (\ref{mainq}) holds for any negative  subharmonic functions in $D\subset \C$. }
\end{remark}

\begin{corollary}\label{cor}
Let $\phi,\psi\in\PSH^-(D)$ with $\phi$ satisfying one of the conditions of Proposition~\ref{propyes}. Then $P[\phi](\psi)=\psi$ if and only if $\psi\le g_\phi$.
\end{corollary}

%\begin{proof}
%Since $P(u,w)\le g_v$ for any $w\in \N_v$, we have $P[v](u)\le P(u,g_v)$. The equality $P[v](u)=u$ gives then $u\le g_v$.
%If $v=g_v+O(1)$, then $g_v\le v+C$ for some $C$ and thus $P(u, g_{v})\le P(u,v+C)\le P[v](u)$. Therefore, the inequality $u\le g_v$ implies %$P[v](u)=u$.
%{\bf Question:} {\sl Is this always true without assuming $g_v=v+O(1)$? This is not easy even for $v$ in the Cegrell class...}
%\end{proof}

Note also that, for functions with small unbounded locus, the following is true.

\begin{proposition}\label{gpas} Let $\phi,\psi\in\PSH_s(D)$, then
 $ g_{P(\phi,\psi)}= g_{P[\phi](\psi)}= g_{P[\psi](\phi)}$.
\end{proposition}

\begin{proof} We have evidently $g_{P(\phi,\psi)}\le g_{P[\phi](\psi)}$. On the other hand, as $C\to\infty$,
$$ g_{P(\phi,\psi)}= {\lim}^* P(P(\phi+C,\psi+C),0) \ge {\lim}^* P(P(\phi+C,\psi),0)
= {\lim}^* P(\phi+C,\psi)={P[\phi](\psi)}.$$
When ${P(\phi,\psi)}\in\PSH_s(D)$, its residual function is idempotent and we derive the reverse inequality.
\end{proof}

\section{Geodesics}\label{sect_geod}
Psh geodesics in the local setting of domains in $\Cn$ were considered in \cite{A}, \cite{BB}, \cite{H}, \cite{R16}.

Denote by $\mathbb A$ the annulus in $\C$
%$\{\zeta\in\C:\:0< \log|\zeta| < 1\}$
bounded by the circles
$T_j=\{\zeta\in\C:\: \log|\zeta|=j\}$, $j=0,1$. Let $D$ be a bounded hyperconvex domain in $\Cn$.
Given two functions $u_0,u_1\in\PSH^-(D)$, equal to zero on $\partial D$, we
consider the class $W=W(u_0,u_1)$
of all functions $u\in\PSH^-(D\times \mathbb A)$ such that $$\limsup_{\zeta\to T_j} u(z,\zeta)\le u_j(z)\quad \forall z\in D.$$
Its Perron envelope $\cP_W(z,\zeta)=\sup\{u(z,\zeta):\: u\in W\}\in W$
satisfies $\cP_W(z,\zeta)=\cP_W(z,|\zeta|)$, which gives rise to the functions $$u_t(z):=\cP_W(z,e^t), \quad 0<t<1,$$ and the map $t\mapsto u_t$ is the  {\it geodesic} for $u_0$ and $u_1$. When $u_t$ tends to $u_j$ as $t\to j$, we say that the geodesic connects $u_0$ and $u_1$.

When the functions $u_j$ are bounded, we have $(dd^c\cP_W)^{n+1}=0$ in $D\times\mathbb A$, and the geodesic $u_t\to u_j$ uniformly on $D$.
In particular, this is true if $u_j$ belong to the Cegrell class $\E_0(D)$ of bounded psh functions $\phi$ in $D$ with zero boundary values on $\partial D$ and $(dd^c\phi)^n(D)<\infty$. By approximation arguments, this extends to functions from the Cegrell class $\F_1(D)$, however in this case the convergence $u_t$ to $u_j$ is in capacity \cite{R16}.

Let now $u_0,u_1$ be arbitrary functions from $\PSH^-(D)$. By \cite{Ce09}, they are limits of decreasing sequences of functions $u_{j,N}\in\E_0(D)$ as $N\to\infty$. Then the corresponding geodesics $u_{t,N}$ decrease to the geodesic $u_{t}$ such that $u_{\Ree\zeta}(z)\in \PSH^-(D\times \A)$. If $u_0,u_1\in\E(D)$, then $u_{t}\in\E(D)$ for any $t$ and $\cP_W(z,\zeta)=u_{\log|\zeta|}(z)\in\E(D\times \A)$ just because $u_t\ge u_0+u_1$, and  $(dd^c\cP_W)^{n+1}=0$. In the general case, $\cP_W$ is still a maximal plurisubharmonic function in $D\times \A$ as the limit of a decreasing sequence of maximal functions.

We are interested in the behaviour of $u_{t}$ as $t\to j\in\{0,1\}$.
Since $u_{t,N}\le (1-t)u_{0,N}+tu_{1,N}$ for any $N$, we get $u_{t}\le (1-t)u_{0}+tu_{1}$. Therefore, $\limsup_{t\to j} u_{t}\le u_j$. Moreover, for any $\epsilon>0$, the capacity of the set $\{z:\: u_{t}(z)> u_0(z)+\epsilon \}$ tends to $0$ as $t\to 0$, and similarly when $t\to 1$. More nontrivial is control over the size of the sets $\{z:\: u_{t}(z)< u_0(z)-\epsilon \}$, which we will handle here by following the proof of \cite[Thm. 5.2]{Da14a} (see Section~5 of \cite{R16} for the affine case).

\begin{theorem}\label{prop_geod} Let $u_0,u_1\in\PSH^-(D)$, then
the geodesic $u_{t}$ converges to $u_0$ in $L_{loc}^1(D)$ (and in capacity) as $t\to 0$  if and only if $P[u_1](u_0)=u_0$.
\end{theorem}

\begin{proof} Denote $ p_C:= P(u_0,u_1+C)$. For any real $C$, the function
$ w_{t,C}=p_C-Ct$ is a subgeodesic for $u_0$ and $u_1$ and so, $w_{t,C}\le u_{t}$ . Therefore, for any $\epsilon>0$,
$$A_{\epsilon,t}:=\{z:\: u_{t}(z)-u_0(z)<-\epsilon\}\subset \{z:\: w_{t,C}(z)-u_0(z)<-\epsilon\}$$
and so,
$$\lim_{t\to 0}\Capa A_{\epsilon,t}\le \inf_{C\ge 0}\Capa B_{\epsilon,C},$$
where $B_{\epsilon,C}=\{z:\: p_C(z)-u_0(z)<-\epsilon\}$.

The family $p_C $ increases as $C\to\infty$  to $P[u_1](u_0)\le u_0$ q.e. and thus both in $L_{loc}^1(D)$ and in capacity. Therefore,  the equality $P[u_1](u_0)= u_0$ implies the convergence of $u_{t}$ to $u_0$ in capacity and in $L_{loc}^1(D)$.

The converse statement is proved in \cite[Thm. 5.2]{Da14a} for quasi-psh functions on compact K\"{a}hler manifolds, however the proof in the local setting is exactly the same. It is based on the relations
$$ u_0(x)=\lim_{\tau\to-\infty}\inf_{0<t<1}(u_t(x)-\tau\, t) = \lim_{\tau\to-\infty}P(u_0, u_1-\tau)(x),$$
valid for almost all $x\in D$, and showing that the convergence $u_t\to u_0$ in $L_{loc}^1(D)$ implies $P[u_1](u_0)=u_0$.
\end{proof}

\medskip

As a direct consequence, we get that any $\phi\in\PSH^-(D)$ can be connected with its Green-Poisson function $g_\phi$ by the geodesic.

\begin{corollary}\label{corind}
Let $u_0=\phi\in\PSH^-(D)$ and $u_1=g_\phi$, then $u_t\to u_j$ in capacity as $t\to j$, $j=0,1$.
\end{corollary}

\begin{proof} The equality $P[g_\phi](\phi)=\phi$ is obvious, while $P[\phi](g_\phi)=g_\phi$ is because, by Proposition~\ref{PuP},
$$g_\phi\ge P(\phi+C,g_\phi)= P(P(\phi+C,0),g_\phi) =  P(\phi+C,0), \quad C>0,$$
and the right hand side converges q.e. to $g_\phi$ as $C\to\infty$.
\end{proof}

\medskip

Another consequence is a necessary condition for connecting psh functions by geodesics.

\begin{corollary}\label{corind1}
No pair of psh functions with different Green-Poisson functions and small unbounded loci can be connected by a geodesic.
\end{corollary}

\begin{proof} This follows from by Theorem~\ref{prop_geod} and Proposition~\ref{onedir}.
\end{proof}

\medskip

Finally, combining Theorem~\ref{prop_geod} with Corollary~\ref{cor}, we get

\begin{theorem}\label{last}
Let $u_j\in \PSH^-(D)$, $j=0,1$, satisfy the conditions on $\phi$ in Proposition~\ref{propyes}. Then $u_t\to u_j$ in capacity as $t\to j$, $j=0,1$, if and only if $u_0\le g_{u_1}$ and $u_0\le g_{u_1}$. When $u_j\in\PSH_s^-(D)$ or $u_j\in\N(D,\fb\phi_j)$ with
$\fb u_j\in\PSH_s^-(D)$, $j=0,1$, this is equivalent to $g_{u_0}= g_{u_1}$.
\end{theorem}

\section{Open questions}
Here we list a few problems that need further investigation. Some of them concern possibility of extending the results from functions with small unbounded loci to general psh functions, while others are not answered even in the case of small unbounded locus and/or functions from the Cegrell class $\E$.

\medskip

I. {\sl Idempotency.} We have established the idempotency $g_{g_\phi}=g_\phi$ of the Green-Poisson functions when $\phi\in\PSH_s^-(D)$ or $\phi\in\N(D,\fb \phi)$ with $\fb\phi\in \PSH_s^-(D)$ (in particular, for $\phi\in\F(D)$). {\it It would be interesting to know if it holds true for any $\phi\in\PSH^-(D)$.}

\medskip
II. {\sl Residual functions of maximal psh functions.} Evidently, $g_\phi\le P(g_\phi^o,g_\phi^b)$ for any psh $\phi$. We have seen that $g_H=g_H^b$ if $H\in\E(D)$ is maximal, so  $g_H^o\ge g_H^b$ in this case. Of course, $g_H^o$ need not to be equal to $0$ for arbitrary maximal $H$ when $n>1$ (take $\phi=\log|f|$ for a holomorphic function $f$ with zeros in $D$), however we wonder  {\it if $g_H^o=0$ if, in addition to the maximality, $H\in\E(D)$.} This would give us $g_\phi^o\in\N(D)$ for any $\phi\in \N(D,\fb \phi)$.

 Also, {\it in the same assumptions, is $g_H$ idempotent?} This would establish the idempotency of $g_\phi$ for functions $\phi\in \N(D,\fb \phi)$.

 \medskip

III. {\sl Singularities of rooftops.} If $g_\phi=g_\psi=g$ is idempotent, then, by Proposition~\ref{gmax}(i), $g_{\max\{\phi,\psi\}}=g$. {\it Is it true that it also implies $g_{P(\phi,\psi)}=g$?}

More generally: {\it Is the relation
\beq\label{singroof} g_{P(\phi,\psi)}= g_{P(g_\phi,g_\psi)}\eeq
true for any $\phi,\psi$ (or, at least, for those with idempotent residual functions)?} Note that, by Proposition~\ref{prG&P}(ii), the right hand side of (\ref{singroof}) equals $P(g_\phi,g_\psi)$, and its left hand side equals $g_{P[\phi](\psi)}$ by Proposition~\ref{gpas} (again, in the idempotent case).

\medskip

IV. {\sl Asymptotic rooftops.} {\it Is it true that
\beq\label{main1} P[\phi](\psi)=P(\psi,g_\phi)\eeq for all $\phi$ and $\psi$?} Except for the cases listed in Proposition~\ref{propyes} and Remark~\ref{remyes}, this is not known even for  $ \E(D)\cap\PSH_s(D)$. It would also prove (\ref{singroof}).

Furthermore, since $P(u,v)=P(u, P(v,0))$ for any $u\in\PSH^-(D)$ and $v\in\PSH(D)$, so we can ask,
more generally: {\it Is it true that $P[\phi_j](\psi)$ converge q.e. to $P[\phi](\psi)$ if $\phi_j$ increase q.e. to $\phi\in\PSH^-(D)$?}
For $\psi=const$ this is true, provided $\phi, \phi_j\in\E(D)$ have finite total Monge-Amp\`ere mass and $\fb\phi\in\PSH_s(D)$, see Theorem~\ref{gincr}. Apart from this, it seems to be unknown even in dimension $1$.

\medskip

V. {\sl Residual second term.} We end with asking about how far, in the psh sense, can a non-model singularity be from its residual function. Denote $r_\phi:=P(\phi-g_\phi)\in\PSH^-(D)$, the {\it residual second term}. If $\phi$ has model singularity, then $r_\phi\in L^\infty(D)$, so $g_{r_\phi}=0$. {\it What can be said about $r_\phi$ in general? Is it true that $g_{r_\phi}=0$ for any psh $\phi$?} If yes, this would prove (\ref{main1}).

More generally: {\it If $\phi_j\nearrow \phi$, is it true that $P(\phi_j-\phi)\nearrow 0$?}

\medskip

\noindent \textbf{Acknowledgements} The author is grateful to the anonymous referees whose valuable suggestions have improved the presentation.

\bigskip\noindent
{\sc Tek/Nat, University of Stavanger, 4036 Stavanger, Norway}

\noindent
\emph{e-mail:} alexander.rashkovskii@uis.no

\end{document}